%% file: 160426a_cutlenses7.tex
\documentclass[12pt,letterpaper]{article}

\usepackage{geometry} 
\usepackage{microtype}
\usepackage{mathtools}
\usepackage{color,latexsym,amssymb,amsthm,graphicx,subfigure}
\usepackage{cite}

\usepackage[small]{titlesec}

\DeclareMathOperator{\polylog}{polylog}

\usepackage{fullpage}

\def\A{\mathcal{A}}
\def\C{\mathcal{C}}

\def\P{\mathcal{P}}

\let\eps\varepsilon
\def\reals{\mathbb{R}}

\newcommand{\CC}{\mathbb{C}}
\newcommand{\RR}{\mathbb{R}}

\newtheorem{theorem}{Theorem}[section]
\newtheorem{lemma}[theorem]{Lemma}
\newtheorem{corollary}[theorem]{Corollary}

\newtheorem*{incidencesPtsCurvesThm}{Theorem \ref{incidencesPtsCurves}}
\newtheorem*{cutSpaceCurvesLem}{Lemma \ref{cutSpaceCurves}}
\newtheorem*{levelsInArrThm}{Theorem \ref{theo:level}}

\newtheorem*{manyfacesCurvesWeakBdThm}{Theorem \ref{manyfacesCurvesWeakBd}}

\theoremstyle{remark}
\newtheorem{remark}{Remark}[section]

\newtheorem{definition}{Definition}[section]

\begin{document}

\title{Cutting Algebraic Curves into Pseudo-segments and Applications }

\author{
  Micha Sharir\thanks{%
    Blavatnik School of Computer Science, Tel Aviv University, Tel-Aviv 69978,
    Israel; \textsl{michas@post.tau.ac.il}. Supported by Grant 2012/229 from the
    U.S.-Israel Binational Science Foundation, by Grant 892/13 from
    the Israel Science Foundation, by the Israeli Centers for Research
    Excellence (I-CORE) program (center no.~4/11), and by the Hermann
    Minkowski--MINERVA Center for Geometry at Tel Aviv University.} 
  \and
  Joshua Zahl\thanks{%
  Department of Mathematics,
    MIT, Cambridge, MA, USA;
    \textsl{jzahl@mit.edu}. Supported by an NSF postdoctoral fellowship.} 
}

\maketitle

\begin{abstract}
We show that a set of $n$ algebraic plane curves of constant maximum degree can be cut into 
$O(n^{3/2}\operatorname{polylog} n)$ Jordan arcs, so that each pair of arcs intersect at most once, i.e., they form a collection of \emph{pseudo-segments}. This extends a similar (and slightly better) bound for pseudo-circles due to Marcus and Tardos. Our result is based on a technique of Ellenberg, Solymosi and Zahl that transforms arrangements of plane curves into arrangements of space curves, so that lenses (pairs of subarcs of the curves that intersect at least twice) become vertical depth cycles. We then apply a variant of a technique of Aronov and Sharir to eliminate these depth cycles by making a small number of cuts, which corresponds to a small number of cuts to the original planar arrangement of curves. After
these cuts have been performed, the resulting curves form a collection of pseudo-segments.
  
Our cutting bound leads to new incidence bounds between points and constant-degree algebraic curves. The conditions for these incidence bounds are slightly stricter than those for the current best-known bound of Pach and Sharir; for our result to hold, the curves must be algebraic and of bounded maximum degree, while Pach and Sharir's bound only imposes weaker, purely topological constraints on the curves. However, when our conditions hold, 
the new bounds are superior for almost all ranges of parameters. We also obtain new bounds on the complexity of a single level in an arrangement of constant-degree algebraic curves, and a new bound on the complexity of many marked faces in an arrangement of such curves. 
\end{abstract}

\section{Introduction}
\label{sec:intro}

Let $\Gamma$ be a finite set of curves in $\reals^2$. The \emph{arrangement} $\A(\Gamma)$ of $\Gamma$
is the planar subdivision induced by $\Gamma$. Its vertices are the intersection points and the endpoints
of the curves of $\Gamma$, its edges are the maximal (relatively open) connected subsets of curves in $\Gamma$ 
not containing a vertex, and its faces are maximal (open) connected subsets of $\reals^2 \setminus \bigcup_{\gamma\in\Gamma}\gamma$. 
Because of their rich geometric structure and numerous applications, arrangements of curves, and especially of
lines and segments, have been widely studied. See \cite{SA} for a comprehensive survey.

A \emph{Jordan arc} is the homeomorphic image of the open interval $(0,1)$\footnote{sometimes in the literature a Jordan arc is defined to be the homeomorphic image of the closed interval $[0,1]$. In this paper, however, we will always use open intervals}; unless otherwise specified, all 
such arcs will be in $\RR^2$. We say that a set $\Gamma$ of Jordan arcs is a set of \emph{pseudo-segments} 
if every pair of arcs in $\Gamma$ intersect at most once, and the arcs cross properly at the point of intersection. Note that in this paper, pseudo-segments may be unbounded. Many combinatorial results on arrangements of lines or segments extend to arrangements of pseudo-segments. 
Three notable examples are (i) the complexity of a single level in an arrangement, (ii) the number 
of incidences between points and curves in the arrangement, and (iii) the complexity of many (marked) 
faces in an arrangement; see, e.g., \cite{AAS,Ch,Sze}. 

However, when two curves are allowed to intersect more than once, the resulting complexity bounds become
weaker. One strategy to address this issue is to cut each curve into several pieces so that the resulting pieces form a collection
of pseudo-segments, and then apply the existing bounds for pseudo-segments to the resulting collection. 
If each pair of curves intersect at most $E$ times, then it is always possible to cut $n$ such curves into 
at most $En^2$ pieces, so that each pair of pieces intersect at most once. When one does this, however, 
the resulting complexity bounds for problems (i)--(iii) are generally poor. In order to obtain 
better bounds, one must cut the curves into fewer pieces. 

This strategy has been pursued successfully for the past 15 years; see~\cite{AAS,ANPPSS,ALPS,AS,lilach,Ch,Ch2,PiSm,TT}. 
However, previous work has almost exclusively focused on arrangements where each pair of curves can intersect at most 
twice (sets of curves of this type are called \emph{pseudo-circles}, or, if unbounded, \emph{pseudo-parabolas}). The current best result in this direction is 
the work of Agarwal et al.~\cite{ANPPSS}, supplemented by that of Marcus and Tardos \cite{MT}. They showed that when $\Gamma$ 
is a set of $n$ pseudo-circles, it is possible to cut the curves of $\Gamma$ into a set of $O(n^{3/2}\log n)$
pseudo-segments.  There are only a few (and considerably weaker) results of this kind for more general families of curves; they include 
works by Chan~\cite{Ch,Ch2} and by Bien~\cite{lilach}. 

In the present paper we study algebraic curves (or more generally, connected subsets of algebraic curves) of constant 
maximum degree. Pairs of such curves might intersect many times---by B\'ezout's theorem, they might intersect as many as $D^2$ times, where $D$
is the maximum degree of the curves. Our main result is a new technique for cutting the curves in such a 
set into a relatively small number of Jordan arcs, each pair of which intersect at most once. 
Our method only applies to algebraic curves (or slightly more generally, connected subsets of algebraic curves), 
but it works well no matter how many times the curves intersect (in brief, the bounds in our results become weaker, but only very slowly, as the degree of the curves increases).

Let $\C$ be a set of algebraic plane curves, no two of which share a common component. Let $\Gamma_0$ be a set of Jordan arcs, each pair of which have finite intersection. We say that $\Gamma_0$ is a \emph{cutting}\footnote{%
  Not to be confused with the notion of $(1/r)$-cutting, which is a decomposition of the plane
  induced by the given curves; see, e.g., \cite{mat-book}.} of $\C$ if each 
curve in $\C$ can be expressed as a finite union of arcs from $\Gamma_0$ plus finitely many points 
(the points at which the original curves are cut). Similarly, let $\Gamma$ be a set of Jordan arcs, each of which is contained in a plane curve, and each pair of which have finite intersection. 
A set $\Gamma_0$ of Jordan arcs is said to be a \emph{cutting} of $\Gamma$ if each 
curve in $\Gamma$ can be expressed as a finite union of arcs from $\Gamma_0$ plus finitely many points. (It is possible to write down a single definition of a cutting for a collection of
curves that contains both of the previous definitions as special cases, but this definition is rather technical so we will not do so here.)

We can now state our main result.
\begin{theorem}[Cutting algebraic curves into pseudo-segments]\label{cuttingCurvesIntoSegments}
Let $\mathcal C$ be a set of $n$ algebraic plane curves of degree at most $D$, no two of which share a common component. Then $\C$ 
can be cut into\footnote{%
  We use the standard notation $O_{\kappa}(\cdot)$ to refer to a constant of proportionality that depends on the parameter
  or parameters $\kappa$.} 
$O_D(n^{3/2} \log^{O_D(1)}n)$ Jordan arcs, so that each pair 
of arcs intersect in at most one point.
\end{theorem}

The above theorem uses the fact that there are at most $O_D(n^2)$ pairwise intersections amongst the curves in $\mathcal C$. 
While this serves as a general upper bound, the actual number of intersections might be much smaller. The following theorem provides a refined bound that depends on the actual number of intersections. 
It is stated in a more general setup that involves Jordan arcs contained in algebraic curves rather than the entire algebraic curves themselves.

\begin{theorem}[Cutting algebraic arcs into pseudo-segments]\label{cuttingArcsIntoSegments}
Let $\Gamma$ be a set of $n$ Jordan arcs, each of which is contained in an algebraic curve of degree at most $D$, and every pair of which have finite intersection. Let 
$X=\sum_{\substack{\gamma,\gamma^\prime\in\Gamma\\\gamma\neq\gamma^\prime}}|\gamma\cap\gamma^\prime|$ 
be the number of times pairs of curves from $\Gamma$ intersect. Then $\Gamma$ can be cut into
$O_D(n + X^{1/2} n^{1/2} \log^{O_D(1)}n)$ Jordan arcs, so that each pair of arcs intersect 
in at most one point. In the worst case, the bound is $O_D(n^{3/2} \log^{O_D(1)}n)$
(as in Theorem~\ref{cuttingCurvesIntoSegments}).
\end{theorem}

\begin{remark}\label{JordanArcsVsCurvesRem}
Since each algebraic curve of degree at most $D$ can be cut into $\leq D(D-1)$ pairwise disjoint Jordan arcs by removing all points at which the curve is singular or is tangent to a vertical line (see Lemma \ref{cuttingCurveJordanArcs} below), the requirement that the curves in $\Gamma$ be Jordan arcs is not a serious constraint, and we only impose it to simplify notation and readability. At the cost of introducing messy notation, we could instead formulate the theorem in a superficially more general fashion by requiring that the curves in $\Gamma$ be open connected subsets of degree $D$ curves, rather than Jordan arcs contained in degree $D$ curves. In particular, Theorem~\ref{cuttingArcsIntoSegments} is indeed a generalization of Theorem \ref{cuttingCurvesIntoSegments}. 
\end{remark}

\begin{remark}
Figure \ref{pseudocubics} depicts a set of $n$ ``pseudo-cubics'' (i.e., Jordan arcs, each pair of which 
intersect at most three times) that requires a quadratic number of cuts in order to turn it into a set 
of pseudo-segments. This demonstrates that in order to obtain sub-quadratic bounds on the number of such cuts, one must 
impose additional restrictions on the given family of curves. For example, Theorems
\ref{cuttingCurvesIntoSegments} and \ref{cuttingArcsIntoSegments} do this by requiring that 
the curves be subsets of bounded-degree algebraic curves.

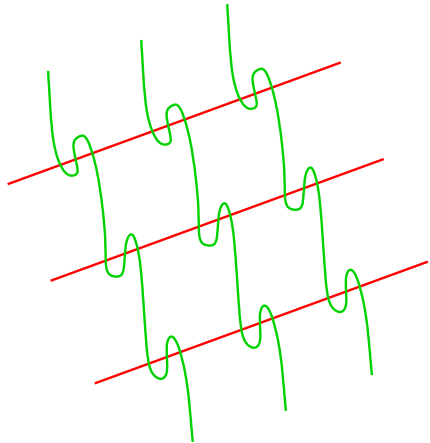
\begin{figure}[hbpt]
  \centering
  \input{cut3int.pstex_t}
  \caption{A set of $n$ pseudo-cubics that require $\Omega(n^2)$ cuts to turn them into pseudo-segments. 
  See Tamaki and Tokuyama~\cite[Theorem 5.3]{TT}.} 
  \label{pseudocubics}
\end{figure}
\end{remark}

\subsection{Point-curve incidences}
Theorem \ref{cuttingArcsIntoSegments} can be applied to obtain new incidence 
theorems in the plane. Pach and Sharir \cite{PS} proved that a set $\P$ of $m$ points and a set $\Gamma$
of $n$ plane curves (either Jordan arcs or bounded degree algebraic curves) determine $O_{s,t}(m^{\frac{s}{2s-1}}n^{\frac{2s-2}{2s-1}}+m+n)$ incidences, provided that 
every pair of curves of $\Gamma$ intersect at most $t$ times, and that there are at most $t$ curves of 
$\Gamma$ passing through any $s$-tuple of points of $\P$ (if the curves are algebraic plane curves, then the implicit constant also depends on the degree of the curves). Sets $\P$, $\Gamma$ of this type are said to 
have $s$ \emph{degrees of freedom} (the parameter $t$ is often suppressed, since as long as it is bounded, independently of $m$ and $n$, it only affects the implicit constant in the incidence bound).

In the case where $\Gamma$ consists of algebraic curves, we will obtain a slightly stronger bound 
under a related (though slightly different) condition. Rather than requiring $\Gamma$ and $\P$ to
have $s$ degrees of freedom, we will assume that the curves of $\Gamma$ lie in an ``$s$-dimensional 
family of curves,'' a notion that we will make precise in Section~\ref{incidencesSection} (see Definition \ref{defnFamilyOfCurves}). 
Roughly, this means that we can represent each curve of $\Gamma$ by a point that lies in some $s$-dimensional
algebraic variety in a suitable parameter space.
For the vast majority of incidence problems that arise in practice, whenever an arrangement of 
algebraic curves has $s$ degrees of freedom, these curves belong to a family of curves of dimension at most $s$. The relationship between having $s$ degrees of freedom and being contained in a family of dimension $s$ is discussed further in Appendix \ref{dimFamilyVsDegreesOfFreedom} below.

Using Theorem \ref{cuttingCurvesIntoSegments}, we can improve the Pach--Sharir bound under the assumptions made above, which hold for a large class of point-curve configurations.

\begin{theorem}[Incidences between points and algebraic curves]\label{incidencesPtsCurves}
Let $\C$ be a set of $n$ algebraic plane curves that belong to an $s$-dimensional family of curves, no two of which share a common irreducible component. Let $\P$ be a set of $m$ points in the
plane. Then for any $\eps>0$, the number $I(\P,\C)$ of incidences
between the points of $\P$ and the curves of $\C$ satisfies
\begin{equation*}
 I(\P,\C) = O\Big(m^{\frac{2s}{5s-4}} n^{\frac{5s-6}{5s-4}+\eps}\Big) + O_{D}\Big(m^{2/3}n^{2/3} + m + n\Big).
\end{equation*}
The implicit constant in the first term depends on $\epsilon$, $s$, the maximum degree of the curves, and also the ``complexity'' of the family of curves from which the set $\C$ is selected. 
\end{theorem}
In Section \ref{incidencesSection} we will give the precise definition of a $s$-dimensional family of curves, and in Section \ref{newIncidenceBoundsSection} we will  state a more rigorous version of Theorem \ref{incidencesPtsCurves} that describes how the implicit constant in the first term depends on the family of curves. 

\subsection{The complexity of a single level in an arrangement}\label{complexityOfALevelSec}

Given a collection $\Gamma$ of algebraic curves, or subsets of such curves,
the {\em level} of a point $p=(x_0,y_0)\in\reals^2$ with respect to $\Gamma$ is defined to 
be the number of intersection points between the downward vertical ray $\{(x_0,y)\in\RR^2 \mid y<y_0\}$ and the curves of $\Gamma$, counted with multiplicity (we assume that each curve in $\Gamma$ has finite intersection with every vertical line).
For each non-negative integer $k$, the {\em $k$-level} of $\A(\Gamma)$ is the closure of the locus of all points on the 
curves of $\Gamma$ whose level is exactly $k$. The $k$-level consists
of subarcs of curves from $\Gamma$ that are delimited 
either at vertices of $\A(\Gamma)$ or at points that lie above a locally $x$-extremal point of
some curve from $\Gamma$. The \emph{complexity} of the $k$-level is the number of subarcs that comprise the level.  

Combining the bound from Theorem~\ref{cuttingArcsIntoSegments} with a result of
Chan~\cite[Theorem 2.1]{Ch}, we obtain the following result. 

\begin{theorem}
\label{theo:level}
Let $\Gamma$ be a set of $n$ Jordan arcs, each of which is contained in an algebraic curve of degree at most $D$, and every pair of which have finite intersection. Then each level of $\A(\Gamma)$ has complexity
$O_D(n^{5/3}\log^{O_D(1)} n)$. 
\end{theorem}

Theorem \ref{theo:level} is proved in Section \ref{subsec:level}. It improves earlier results of Chan~\cite{Ch,Ch2} and Bien~\cite{lilach}
for the general algebraic case, and it almost matches the results
in \cite{ANPPSS,MT} for the case of pseudo-circles and pseudo-parabolas.

\subsection{Complexity of many marked faces in an arrangement}\label{complexityMarkedFacesIntroSec}

Let $\Gamma$ be a set of $n$ Jordan arcs, each of pair of which has finite intersection. Let $\P$ be a set of $m$ points in the plane with the property that no point of $\P$ lies on any curve of $\Gamma$. We define $K(\P,\Gamma)$ to be the sum of the complexities of the faces of $\A(\Gamma)$ that contain at least one point of $\P$, where the complexity of a face is the number of edges of $\A(\Gamma)$ 
on its boundary. Informally, this can be regarded as  an ``off-curve'' incidence
question, where instead of counting the number of curves each point intersects,
we (more or less) count the number of curves that the point can ``reach'' (without crossing other curves). The 
problem has been studied in the context of lines (see~\cite{SA}), segments, and
circles~\cite{AAS,ANPPSS,AS}. 

We will establish the following bound on the complexity of many marked faces:
\begin{theorem}[Complexity of many faces]\label{manyfacesCurvesWeakBd} 
Let $\Gamma$ be a set of $n$ Jordan arcs, each of which is contained in an algebraic curve of degree at most $D$, and every pair of which have finite intersection. Let $\P$ be a set of $m$ points in the plane, so that no point of $\P$ lies on any curve of $\Gamma$. Then 
\begin{equation} \label{weakptfaces}
K(\P,\Gamma) = O_D(m^{2/3}n^{2/3}+n^{3/2}\log^{O_D(1)}n).
\end{equation} 
\end{theorem}

Theorem \ref{manyfacesCurvesWeakBd} is obtained by using Theorem \ref{cuttingArcsIntoSegments} to cut the Jordan arcs into pseudo-segments and then applying existing techniques to this
collection of pseudo-segments. As discussed in Remark \ref{JordanArcsVsCurvesRem}, we could also state Theorem \ref{manyfacesCurvesWeakBd} for collections of algebraic curves (or collections
of connected subsets of algebraic curves) rather than Jordan arcs contained in algebraic curves. Doing so, however, makes the notation more complex without actually making the result any more general.

We prove Theorem \ref{incidencesPtsCurves} by using Theorem \ref{cuttingCurvesIntoSegments} to obtain a weak incidence bound and then amplifying the bound using further arguments. The bound in Theorem \ref{manyfacesCurvesWeakBd} is the analogue of the weak incidence bound which is the starting point for the proof of Theorem \ref{incidencesPtsCurves}. We attempted to amplify the bound in Theorem \ref{incidencesPtsCurves} as well, but we encountered several technical issues that we do not know how to overcome. In Section \ref{markedFacesDiscussionSec} we will comment
on these difficulties and leave such an improvement as an open problem.


%
%
\section{Cutting algebraic arcs into pseudo-segments} \label{sec:cut}

In this section we prove Theorems \ref{cuttingCurvesIntoSegments} and \ref{cuttingArcsIntoSegments}. 
\subsection{Some real algebraic geometry}\label{realAlgGeoSec}
Before we can proceed further, we will need some basic definitions from real algebraic geometry.  A real (resp., complex) affine algebraic
variety is the common zero locus of a finite set of polynomials over the real (resp., complex) numbers. If $Z\subset\RR^d$ is a real algebraic variety, we define $Z^*\subset\CC^d$ to be the smallest (complex) variety that contains $Z$. Unless otherwise noted, all varieties are assumed to be affine. Throughout the proof, we will work with both the Euclidean and Zariski topology\footnote{See \cite{Harris} for an introduction to the Zariski topology and related background.}. Unless specified explicitly, all open sets are assumed to be in the Euclidean topology.

Let $Z\subset \RR^d$ be a real algebraic variety. A crucial property of $Z$ will be its \emph{dimension}. The precise definition of
the dimension of a real algebraic variety is
slightly subtle, see, e.g., \cite{BCR}. Informally, however, the dimension of $Z$ is the largest integer $e$ so that $Z$
contains a subset homeomorphic to the open $e$-dimensional cube $(0,1)^e$. If $Z\subset\RR^d$ is a non-empty algebraic set of dimension $\leq
1$, we call it an \emph{algebraic curve}. 

The \emph{degree} of $Z\subset\RR^d$ is the degree of the complex variety $Z^*\subset\CC^d$; the latter is the sum of the degrees of the irreducible components of $Z^*$. See \cite{Harris} for further background and details.

Similarly, a real (resp., complex) \emph{projective} algebraic variety is the common zero locus of a finite set of homogeneous
polynomials. The dimension and degree of a real projective variety are defined analogously to the definitions for the affine case.

For a single polynomial $f\in\RR[x_1,\ldots,x_d]$, its zero locus $Z(f)=\{p\in\RR^d \mid f(p)=0\}$ 
is a real algebraic variety of degree at most $\deg(f)$.

\subsection{Cuttings}\label{cuttingsSection}
Let $\mathcal{S}$ be a collection\footnote{The terms ``collection'' and ``set'' mean the same thing; we use both merely to improve readability.} of
sets in $\RR^2$, each of which is contained in an algebraic curve and each pair of which have finite intersection. We say that $\mathcal{S}^\prime$ is a \emph{cutting} of $\mathcal{S}$
if $\mathcal{S}^\prime$ is a collection of pairwise disjoint connected sets, and for each $S\in\mathcal{S}$ there is a finite set $\mathcal{P}_S\subset S $ so that
\begin{equation*}
\mathcal{S}^\prime=\bigcup_{S\in\mathcal{S}}\{S^\prime \mid S^\prime\ \textrm{is a connected component of}\ S\backslash\mathcal{P}_S\}.
\end{equation*}
We say that $\sum_{S\in\mathcal{S}}|\mathcal{P}_S|$ is the number of cuts used in the cutting. Note that if $S^\prime\in \mathcal{S}^\prime$, then
either $S^\prime$ is a point, or there is a unique $S\in\mathcal{S}$ with $S^\prime\subset S$. In practice, we will throw away all isolated points, so
each $S^\prime\in\mathcal{S}$ will have (be contained in) a unique ``parent'' set $S\in\mathcal{S}$.

The following three results will help us control the number of connected components that are obtained by cutting an algebraic plane curve.
\begin{theorem}[Harnack~\cite{Harnack}]\label{HarnackCurveTheorem}
Let $f\in\RR[x,y]$ be a polynomial of degree $D$. Then $Z(f)$ contains at most $\frac12(D-1)(D-2) +1 \leq D^2$ connected components.
\end{theorem}

\begin{lemma}[Removing a point from a curve]\label{removePtFromSet}
Let $f\in\RR[x,y]$ be a polynomial of degree $D$, let $\gamma\subset Z(f)$ be a connected set, and let $p\in\gamma$. Then $\gamma\backslash\{p\}$ contains at most $D$ connected components. 
\end{lemma}
The idea behind Lemma \ref{removePtFromSet} is that in a small neighborhood of $p$, $\gamma$ is a union of at most $D$ ``branches,'' and
removing $p$ can cut these branches into separate connected components. See, e.g., Lemmas 4.5, 4.6, and 4.7 from \cite{Zahl} for details. Note that for the purposes of this paper, the exact bounds in Theorem \ref{HarnackCurveTheorem} and Lemma \ref{removePtFromSet} are not important; all that matters is that the quantities are $O_D(1).$

\begin{lemma}[Cutting a curve into Jordan arcs]\label{cuttingCurveJordanArcs}
Let $f\in\RR[x,y]$ be a square-free polynomial. Then $Z(f)\backslash Z(\partial_y f)$ is a union of disjoint Jordan arcs.
\end{lemma}
\begin{proof}
First, note that by the implicit function theorem, $Z(f)\backslash Z(\partial_y f)$ is a one-dimensional manifold. By the classification of
one-dimensional manifolds, we conclude that each connected component of $Z(f)\backslash Z(\partial_y f)$ is either a Jordan arc (i.e., homeomorphic to
the interval $(0,1)$), or is homeomorphic to a circle. Suppose that a connected component $\gamma\subset Z(f)\backslash Z(\partial_y f)$  is
homeomorphic to a circle. Since $\gamma\subset\RR^2$ is compact, there exists a point $(x_0,y_0)\in\gamma$ with $x_0 = \min\{x\mid (x,y)\in\gamma\}.$
We must have $\partial_y f(x_0,y_0)=0$, which contradicts the fact that $\gamma\subset Z(f)\backslash Z(\partial_y f)$. 
\end{proof}

Lemmas \ref{HarnackCurveTheorem} and \ref{removePtFromSet} imply that when a collection of algebraic curves is cut, the number of sets in the new
collection is controlled by the number of curves in the original collection and the number of cuts made in the cutting. This is made precise in the following lemma.
\begin{lemma}\label{cutsVsComponents}
Let $\mathcal{C}$ be a set of algebraic curves and let $\Gamma_0$ be a cutting of $\mathcal{C}$. Suppose that $\ell$ cuts are used in the cutting. Then $|\Gamma_0|\leq D^2|\C|+D\ell$.
\end{lemma}
%

\subsection{Lenses}
Let $\gamma$ and $\gamma^\prime$ be Jordan arcs. We say that $\gamma$ and $\gamma^\prime$ form a \emph{lens} if
$\RR^2\backslash(\gamma\cup\gamma^\prime)$ consists of at least two connected components. We say that $\gamma$ 
and $\gamma^\prime$ form a \emph{proper lens} if $\RR^2\backslash(\gamma\cup\gamma^\prime)$ consists of exactly two connected components. 
If $\lambda$ and $\lambda^\prime$ form a lens, then we can always find two connected subarcs $\delta\subseteq\gamma$ and $\delta'\subseteq\gamma'$ with two common endpoints. If the lens is proper, then the sets $\delta$ and $\delta^\prime$ are unique, and the relative interiors of $\delta$ and $\delta'$ are disjoint. See Figure~\ref{fig:lenses}. We will abuse notation slightly and will also refer to the pair $(\delta,\delta^\prime)$ as the lens.

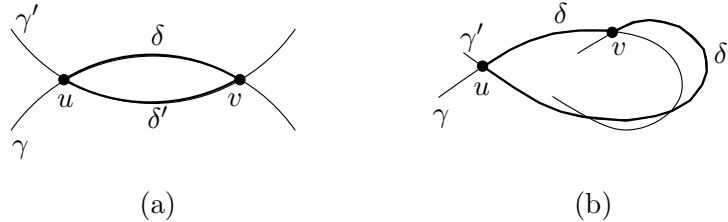
\begin{figure}[hbpt]
  \centering
 \input{figlenses.pstex_t}
  \caption{(a) $\gamma$ and $\gamma'$ form a proper lens that consists of the subarcs $\delta$ and $\delta'$. 
  (b) The lens formed by $\gamma$ and $\gamma'$ with endpoints $u,v$ is not proper.}
  \label{fig:lenses}
\end{figure}

Let $\Gamma$ be a set of Jordan arcs. We say that $\Gamma$ is \emph{lens-free} if no two curves from $\Gamma$ form a lens. $\Gamma$ is lens-free if and only if the curves in $\Gamma$ are a set of pseudo-segments.

\subsection{Lifting plane curves to space curves}
In this section we will describe a process adopted from Ellenberg, Solymosi and Zahl~\cite{ESZ}
that transforms a plane curve into a space curve, so that the ``slope'' of the plane curve is encoded
as the $z$--coordinate of the space curve. If $C$ is a plane curve and $(x,y)$ is a smooth point of $C$, we define the slope of $C$ at $(x,y)$
to be the slope of the tangent line to $C$ at $(x,y)$. If this line is vertical, then we say that the slope is infinite.
\begin{lemma}\label{liftOfACurve}
Let $C$ be an irreducible algebraic curve in $\RR^2$ of degree at most $D$. Then there is an irreducible space curve $\hat C\subset\RR^3$ with the following property: if $(x,y,z)\in \hat C$ and if $(x,y)$ is a smooth point of $C$ where the slope is finite, then $z$ is the slope of $C$ at the point $(x,y)$. Furthermore, the degree of $\hat C$ is at most $D^2$.
\end{lemma}
This is \cite[Proposition 1]{ESZ}. In brief, let $f$ be an irreducible polynomial satisfying $C=Z(f)$, and consider the algebraic variety 
\begin{equation*}
\big\{(x,y,z) \mid f(x,y)=0,\; z\partial_y f(x,y) + \partial_x f(x,y)  = 0 \big\} .
\end{equation*}

As discussed in \cite[\S3.3]{ESZ}, this variety is a union of vertical lines (one line above each singular point of $C$), plus an irreducible curve that is not a vertical line.
This curve is $\hat C$. Its degree is $\le D^2$, since it is an irreducible component of the intersection of two surfaces of degree at most $D$.
See \cite{ESZ} for details.\footnote{%
  Note that regular points of $C$ with vertical tangency are not part of the projection of $\hat C$. For example, if $C$ is the circle
  $x^2+y^2=1$, $\hat C$ is the space curve given by $x^2+y^2=1$ and $x+yz=0$, and its $xy$-projection does not contain the points $(1,0)$ or $(-1,0)$.}

\begin{remark}
Note that if $(x,y)$ is a smooth point of $C$ with finite slope, then the $z$-vertical line passing through $(x,y)$ intersects $\hat C$ in
exactly one point. Thus if $\gamma\subset C$ is a Jordan arc consisting of smooth points with finite slope, then there is a unique space Jordan
arc $\hat \gamma\subset \hat C$ satisfying $\pi(\hat \gamma)=\gamma$, where $\pi(x,y,z)=(x,y)$. 
We will exploit this observation by cutting $C=Z(f)$ at each point where $\partial_yf$ vanishes. By B\'ezout's theorem and Lemma \ref{removePtFromSet}, this process cuts $C$ into $O_D(1)$ connected pieces, so that the interior of each piece consists exclusively
of smooth points with finite slope.
\end{remark}

\subsection{Depth cycles and lenses}

Let $\gamma$ and $\gamma^\prime$ be Jordan space arcs in $\RR^3$. We say that $\gamma$ and $\gamma^\prime$ form a \emph{depth cycle of length two} if there are
points $(x_1,y_1,z_1),(x_2,y_2,z_2)\in\gamma,$ $(x_1,y_1,z_1^\prime),(x_2,y_2,z_2^\prime)\in\gamma^\prime$, so that $(x_1,y_1)\neq (x_2,y_2)$, $z_1\geq z_1^\prime$, and $z_2\leq z_2^\prime.$
This depth cycle is characterized by the tuple $(\gamma,\gamma^\prime, x_1,y_1,z_1,z_1^\prime, x_2,y_2,z_2,z_2^\prime)$. If $z_1>z_1^\prime$ and $z_2<z_2^\prime$, we call the depth cycle a
\emph{proper} depth cycle (of length two).

Let $\Gamma$ be a set of Jordan space arcs in $\RR^3$. We say that $\Gamma$ has no depth cycles 
of length two (resp., no proper depth cycles of length two) if no pair of curves in $\Gamma$ have a depth cycle of length two (resp., a proper depth cycle of length two).

\begin{lemma}\label{lensesAndDepthCycles}
Let $C$ and $C^\prime$ be plane algebraic curves. Let $\gamma\subset C$ and $\gamma^\prime\subset C^\prime$ be $x$-monotone Jordan arcs consisting of smooth points with finite slope, and suppose that $\gamma$ and $\gamma^\prime$ form a lens. Then $\hat \gamma$ and $\hat \gamma^\prime$ form a depth cycle of length two. 
\end{lemma}
\begin{proof}
By shrinking $\gamma$ and $\gamma^\prime$ if necessary, we can assume that $\gamma$ and $\gamma^\prime$ intersect at exactly two points 
(since the arcs are $x$-monotone, this is equivalent to them forming a proper lens)---call these points $p$
and $q$. In particular, $\RR^2\backslash (\gamma\cup\gamma^\prime)$ has exactly two connected components, exactly one of which is unbounded. Call the bounded component the ``inside'' of
$\gamma\cup\gamma^\prime$. Since $\gamma$ and $\gamma^\prime$ are $x$-monotone, every $y$-vertical line intersects each of $\gamma$ and $\gamma^\prime$ at most once. In particular, every vertical line that intersects the inside of $\gamma\cup\gamma^\prime$ must intersect each of $\gamma$ and $\gamma^\prime$ at precisely one point. By renaming the indices if necessary, we can assume that if $\ell$ is a vertical line that intersects the inside of $\gamma\cup\gamma^\prime$, then the intersection of $\ell$ with $\gamma$ has larger $y$-coordinate than the intersection of $\ell$ with $\gamma^\prime$. 

\begin{figure}[hbpt]
  \centering
 \input{lift.pstex_t}
  \caption{(a) $\gamma$ and $\gamma'$ form a lens in the $xy$-plane. (b) The respective lifted images $\hat{\gamma}$, $\hat\gamma'$ of $\gamma$, $\gamma'$
  form a depth cycle of length two.}
  \label{fig:lift}
\end{figure}
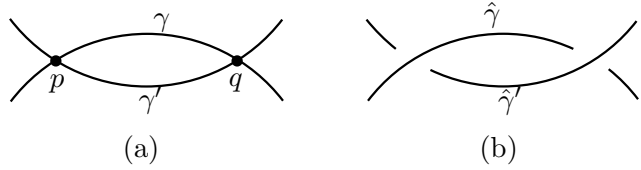

Suppose that $\gamma$ and $\gamma^\prime$ are tangent at $p=(x,y)$. Then $\hat\gamma$ and $\hat\gamma^\prime$ intersect at the point $(x,y,z)$, where
$z$ is the slope of both curves at $p$. By interchanging the indices if necessary, we can assume that at the lifting of $q$, either $\hat\gamma$ and $\hat\gamma^\prime$ intersect, or $\hat\gamma$ has larger $z$-coordinate. In either case, $\hat\gamma$ and $\hat\gamma^\prime$ form a
depth cycle of length two. An identical argument can be used if $\gamma$ and $\gamma^\prime$ are tangent at $q$. 

Suppose then that $\gamma$ and $\gamma^\prime$ are not tangent at $p$ or at $q$. For each $x$ that lies in the $x$-projection of the inside of $\gamma\cup\gamma^\prime$, let $f_1(x)$ (resp., $f_2(x)$) be the $y$-coordinate of the intersection of $\gamma$ (resp., $\gamma'$)
with the vertical line passing through $(x,0)$. Let $f(x)=f_1(x)-f_2(x)$. Let $p=(x_1,y_1)$, $q=(x_2,y_2)$. Then $f(x_1)=f(x_2)=0$, and $f(x)>0$ for
$x_1<x<x_2$. Furthermore, $f(x)$ is smooth, and $\frac{df}{dx}(x_1)\neq 0,$ $\frac{df}{dx}(x_2)\neq 0.$ We conclude that $\frac{df}{dx}(x_1)>0$ and
$\frac{df}{dx}(x_2)<0$, i.e., the slope of $\gamma$ at $p$ is larger than that of $\gamma^\prime$, and the slope of $\gamma$ at $q$ is smaller than that of $\gamma^\prime$. We conclude that $\hat\gamma$ and $\hat\gamma^\prime$ form a depth cycle. 
\end{proof}

\subsection{Cutting lenses}
We are almost ready to prove Theorem~\ref{cuttingCurvesIntoSegments}. Before doing so, we will need the following key lemma.
\begin{lemma}\label{cutSpaceCurves}
For each $D\geq 1$, there are constants $A = A(D)$ and $\kappa=\kappa(D)$ so that the following holds. 
Let $\C$ be a set of $n$ irreducible algebraic plane curves of degree at most $D$ and let $\hat\C=\{\hat C \mid C\in\C\}$. 
Then by using $\leq A n^{3/2} \log^{\kappa}n$ cuts, $\hat \C$ can be cut into a set of Jordan space arcs that have no proper depth cycles of length two. 
\end{lemma}
To avoid interrupting the flow of the proof, we will prove Lemma \ref{cutSpaceCurves} in Appendix \ref{proofOfLemCutSpaceCurvesSec} below. 
A similar statement appears in the recent work of Aronov and Sharir~\cite{ArS1}. That work deals primarily with eliminating cycles
(of any length) in three-dimensional line configurations that satisfy certain genericity assumptions. It also presents an extension of this result
to the case of constant-degree algebraic curves, but does so in a rather sketchy way. For the sake of exposition, and in order
to make the present paper as self contained as possible, we give a detailed and rigorous proof for the specific case that we need.

\begin{remark}
Although we will not need it here, we can define a depth cycle of any length $\ell\geq 2$ in a similar fashion. 
The cutting from Lemma \ref{cutSpaceCurves} will actually eliminate depth cycles of all lengths $\ell\geq 2$. 
\end{remark}

\begin{proof}[Proof of Theorem \ref{cuttingCurvesIntoSegments} using Lemma \ref{cutSpaceCurves}]

The proof of Theorem \ref{cuttingCurvesIntoSegments} will proceed as follows. First, we will cut the curves from $\C$ so as to eliminate all lenses where the corresponding curves intersect
transversely (such lenses correspond to proper depth cycles). Then we will cut the curves from $\C$ so as to eliminate all lenses where the corresponding curves intersect tangentially at one
or both endpoints of the lens. Finally, we will further cut the curves from $\C$ so that the resulting pieces are $x$-monotone and smooth Jordan arcs.

Let $\hat\C=\{\hat C \mid C\in\C\}$. Use Lemma \ref{cutSpaceCurves} to cut $\hat\C$ into Jordan space arcs 
so that all proper depth cycles of length two are eliminated; this cutting uses $O_D\left( n^{3/2} \log^{O_D(1)}n \right)$ cuts (recall that if the same point $p\in\RR^3$ is removed from several curves, then this point is counted with multiplicity). 

The projection of each of these Jordan space arcs to the $xy$-plane yields a connected subset of a curve from $\C$. Let $\mathcal{D}_1$ denote the set of all connected components of the
projections of these Jordan space arcs. Then $\mathcal{D}_1$ is a cutting of $\C$ using
$O_D\left( n^{3/2} \log^{O_D(1)}n \right)$ cuts; the resulting segments do not form any proper lenses.

Next, we further cut the sets in $\mathcal{D}_1$ as follows. For each pair of connected sets $S,S^\prime\in\mathcal{D}_1$ where neither $S$ nor $S^\prime$ is a point, let $C,C^\prime\in\C$ be
the (uniquely defined) curves from $\C$ containing $S$ and $S^\prime$, respectively. Cut $S$ and $S^\prime$ at those points $p\in S\cap S^\prime$ where $p$ is a smooth point of both $C$ and
$C^\prime$, and $C$ and $C^\prime$ are tangent at $p$---points of this type might be endpoints of an improper lens. After this procedure has been performed for every pair of sets $S,S^\prime\in\mathcal{D}_1$, finitely many points have been removed from each set $S\in\mathcal{D}_1$. The total number of cuts performed at this stage is at most
\begin{align}\label{ptsOfTangency}
2\sum_{C\in\mathcal{C}}\Big|\{p\in C\mid& \ p\ \textrm{is a smooth point of}\ C,\ \textrm{and there exists a }  \\[-15pt]
&  \textrm{curve}\ C^\prime\in\C\ \textrm{that is smooth at}\ p\ \textrm{and tangent to}\ C\ \textrm{at}\ p\}\Big|. \nonumber
\end{align}
By \cite[Theorem 1]{ESZ}, the sum in \eqref{ptsOfTangency} is $O_D(n^{3/2})$. 

Finally, we further cut the sets in $\mathcal{D}_2$ as follows. For each $S\in\mathcal{D}_2$ that is not a point, let $C=Z(f)$ be the unique curve from $\C$ containing $S$, 
for a suitable bivariate polynomial $f$ of degree at most $D$. Remove from $S$ those points satisfying $\partial_yf=0$. If $S$ is a point, remove it entirely (such points will be singular points of any algebraic curve that contains them). By B\'ezout's theorem, this process uses $O_D(n)$ cuts. Let $\Gamma_0$ be the collection of connected components of sets from $\mathcal{D}_2$ after this cutting process.

Each of the sets in $\Gamma_0$ is an $x$-monotone Jordan arc, and by Lemma \ref{lensesAndDepthCycles}, each pair of arcs of $\Gamma_0$
intersect at most once. Thus $\Gamma_0$ is a cutting of $\mathcal{C}$ (in the sense of Section \ref{cuttingsSection}) that uses $O_D(n^{3/2}\log^{O_D(1)} n)$ cuts. 
By Lemma \ref{cutsVsComponents}, $|\Gamma_0|=O_D(n^{3/2}\log^{O_D(1)} n)$. Thus $\Gamma_0$ satisfies the conclusions of Theorem \ref{cuttingCurvesIntoSegments}.
\end{proof}

\begin{proof}[Proof of Theorem \ref{cuttingArcsIntoSegments} using Theorem \ref{cuttingCurvesIntoSegments}]
Let $\Gamma$ be a set of $n$ Jordan arcs, each of which is contained in an algebraic curve of degree at most $D$, and every pair of which have finite intersection. Let $X=\sum_{\substack{\gamma,\gamma^\prime\in\Gamma\\\gamma\neq\gamma^\prime}}|\gamma\cap\gamma^\prime|$ 
be the number of times pairs of curves from $\Gamma$ intersect. We will obtain an improved bound when the number of curve-curve intersections is much smaller than $n^2$. 

First, the case where $X=O(n)$ is trivial: we simply cut each arc at all its intersection
points with the other arcs, and get a total of $O(n+X)=O(n)$ pairwise disjoint subarcs; this certainly satisfies the bound in Theorem \ref{cuttingArcsIntoSegments}. Assume then that $X$
is superlinear in $n$. 

\begin{definition}
Let $\Gamma$ be a collection of Jordan arcs in the plane, let $\A(\Gamma)$ be the arrangement determined by $\Gamma$, and let $r\geq 1$. A \emph{$(1/r)$-cutting} of $\A(\Gamma)$ 
into pseudo-trapezoids is a collection $\Xi$ of pairwise-disjoint open connected sets in $\RR^2$ (these sets are called the cells of the cutting), so that the following properties hold
\begin{enumerate}
\item[(i)] Each cell is crossed by at most $n/r$ curves from $\Gamma$ (we say a curve from $\Gamma$ crosses a cell if the curve intersects the cell).
\item[(ii)] The closures of the cells cover the plane.
\item[(iii)] The boundary of each of these cells is the union of at most two vertical line segments and two Jordan arcs, where each arc is a subarc of an arc from $\Gamma$.
\end{enumerate}
\end{definition}

Let $r=\lceil n^2/X\rceil$. Construct a $(1/r)$-cutting $\Xi$ of $\A(\Gamma)$ into pseudo-trapezoids, which consists of 
$O(r+r^2X/n^2) = O(r)$ cells, each of which intersect $n/r=O(X/n)$ curves of $\Gamma$.
The existence of such a cutting has been established by de Berg and Schwarzkopf~\cite{BS} for the case of line segments, and has been
considered as a folklore result for the general case, with an essentially identical proof (see \cite{AAS} and \cite[Proposition 2.12]{HP}).

First, cut each arc of $\Gamma$ at each of its intersection points with the boundaries of the cells of $\Xi$. If an arc of $\Gamma$ occurs as the boundary of one or more cells, cut that arc at
each point where it meets a vertical line segment from the boundary of the trapezoid (these points are the ``corners'' of the trapezoid). This procedure cuts $\Gamma$ into a new collection
$\Gamma_1$ of Jordan arcs, and uses $O(r)\cdot (n/r) = O(n)$ cuts. 

Note that if $\gamma,\gamma^\prime\in\Gamma_1$ form a lens, and if $\delta\subset\gamma,\ \delta^\prime\subset\gamma^\prime$ are Jordan arcs with common endpoints, then $\delta$ and
$\delta^\prime$ must be contained in a common cell from $\Xi$, for otherwise one of them would have to be cut by the procedure just mentioned. Thus, to eliminate all lenses from $\Gamma_1$, it suffices to cut each of the curves in $\Gamma_1$ into smaller Jordan arcs so that within each cell of $\Xi$, all lenses are eliminated.

Each cell $\tau$ of $\Xi$ intersects $O(X/n)$ curves from $\Gamma$; call this collection of curves $\Gamma_{\tau}$. Each curve $\gamma\in\Gamma_\tau$
is contained in a unique algebraic curve $C_{\gamma}$. Let $\C_{\tau}=\{C_\gamma \mid \gamma\in\Gamma_{\tau}\}$. Then $|\C_\tau|\leq|\Gamma_\tau|=O(X/n)$. 

 Apply Theorem \ref{cuttingCurvesIntoSegments} to $\C_\tau$; we obtain a cutting $\Gamma^\prime_{\tau}$ of $\C_\tau$ that uses $O((X/n)^{3/2}+\log^{O_D(1)}(X/n))$ cuts, so that each pair of arcs in $\Gamma^\prime_{\tau}$ intersect at most once. For each curve $C\in C_{\tau},$ let $\mathcal{P}_{C,\tau}$ be the set of points at which $C$ is cut. For each $\gamma\in\Gamma_1$, let $C_\gamma$ be the corresponding algebraic curve and define
 \begin{equation*}
 \mathcal{P}_\gamma =\gamma\cap \bigcup_{\tau \mid \gamma\in\Gamma_\tau} \mathcal{P}_{C_\gamma,\tau}.
\end{equation*}
$\mathcal{P}_\gamma$ is the set of points of $\gamma$ at which $C_\gamma$ is cut, when $C_\gamma$ is regarded as a curve in $\C_\tau$ for some cell $\tau$ containing $\gamma$. 

Let $\Gamma_2$ be the collection of Jordan arcs obtained by cutting each arc $\gamma\in \Gamma_1$ at each point of $\mathcal{P}_\gamma$. The total number of cuts is 
\begin{equation*}
O(n^2/X)O((X/n)^{3/2}\log^{O_D(1)}(X/n))+O(n)=O(n^{1/2}X^{1/2}\log^{O_D(1)}n),
\end{equation*}
and the curves in $\Gamma_2$ satisfy the conclusions of Theorem \ref{cuttingArcsIntoSegments}.
\end{proof}

\section{Point-curve incidences}\label{incidencesSection}

In this section we will prove (a precise version of) Theorem \ref{incidencesPtsCurves}. In order to do so, we 
first define rigorously the notion of a family of algebraic curves.

\subsection{Families of algebraic curves}
Let $f\in\RR[x,y]$ be a polynomial of degree at most $D$. $f$ can be written as a sum of $\binom{D+2}{2}$ monomials (some of which might have zero coefficients),
and thus we can identify $f$ with a vector in $\RR^{\binom{D+2}{2}}$. If $\lambda\neq 0$, then $f$ and $\lambda f$ have the same zero-set. Thus, the set of algebraic curves of degree at most $D$ in $\RR^2$ can be identified with the points in the projective space $\mathbf{P}\mathbb{R}^{\binom{D+2}{2}}$.
Henceforth we will abuse notation and refer to such algebraic curves as elements of $\mathbf{P}\mathbb{R}^{\binom{D+2}{2}}$, and vice-versa.

\begin{definition}\label{defnFamilyOfCurves}
An \emph{$s$-dimensional family of plane curves of degree at most $D$} is an algebraic variety $F\subset  \mathbf{P}\mathbb{R}^{\binom{D+2}{2}}$ that has dimension $s$. We will call the degree of
$F$ the \emph{complexity} of the family. We use the term complexity (rather than degree) to avoid confusion with $D$, which is the maximum degree of the plane curves. Since the degree of the
curves and the complexity of the family are of secondary importance in the context that we consider here, we will sometimes abbreviate this as ``an $s$-dimensional family of plane curves.''
\end{definition}

For example, the set of unit circles in the plane is a two-dimensional family of curves; the set of circles (of arbitrary radius) in the plane is a
three-dimensional family of curves; and the set of axis-parallel ellipses or of hyperbolas are four-dimensional families. In each of
these instances, $D=2$ and the complexity is $O(1)$. 

Informally, if $F$ is $s$-dimensional then we expect to be able to characterize each element of $F$ by $s$ real
parameters. This is the case with all the aforementioned examples and in most of the applications. In this case, requiring a curve of $F$ to pass through $s$ points in the plane imposes $s$ constraints on the $s$ parameters specifying the curve,
and we expect these equations to have a finite number of solutions. When all these expectations are satisfied, we indeed
get a family of curves with $s$ degrees of freedom (as in Pach and Sharir~\cite{PS}). In Appendix \ref{dimFamilyVsDegreesOfFreedom} we further discuss the connection between having $s$ degrees of freedom and belonging to an $s$-dimensional family of curves.

\subsection{New incidence bounds}\label{newIncidenceBoundsSection}
We can now state (and prove) a precise version of Theorem \ref{incidencesPtsCurves}.
\begin{incidencesPtsCurvesThm}[Incidences between points and algebraic curves]
Let $\C$ be a set of $n$ algebraic plane curves that belong to an $s$-dimensional family of curves of complexity $K$, no two of which share a common irreducible component. Let $\P$ be a set of $m$ points in the plane. Then for each $\eps>0$, the number $I(\P,\C)$ of incidences
between the points of $\P$ and the curves of $\C$ satisfies
\begin{equation} \label{newIncidenceBdPrecise}
 I(\P,\C) = O_{s,D,K,\eps}\Big(m^{\frac{2s}{5s-4}} n^{\frac{5s-6}{5s-4}+\eps}\Big) + O_D\Big(m^{2/3}n^{2/3} + m + n\Big).
\end{equation}
\end{incidencesPtsCurvesThm}
\medskip

\begin{remark} 
If the arrangement of points and curves also has $s\ge 3$ degrees of freedom, then Pach and Sharir's bound from \cite{PS} would say
that $I(\P,\C)=O_{D,s}\big(m^{\frac{s}{2s-1}}n^{\frac{2s-2}{2s-1}}+m+n\big)$. The bound \eqref{newIncidenceBdPrecise} is superior 
for $m > n^{1/s + c}$ for any constant $c>0$, with a suitable $\eps$ that depends linearly on $c$. When $m\leq n^{1/s-c}$, both bounds become $O(n)$ (again, with a suitable choice of $\eps$), and when $m$ is close to $n^{1/s}$,
our bound is larger by a factor of $n^{\eps}$ than the bound in \cite{PS}.
\end{remark}

\begin{remark}
For $s=2$, which arises for lines and for unit circles, we almost recover the Szemer\'edi--Trotter bound (we miss by an $n^\eps$ factor). For
$s=3$, which arises for arbitrary circles and for vertical parabolas, we again almost recover the bound 
$O(m^{6/11}n^{9/11}\log^{2/11}n+m^{2/3}n^{2/3}+m+n)$ from~\cite{ANPPSS,AS,MT} (in our bound, the $\log^{2/11}n$ is weakened to $n^\eps$). For $s=4$, which arises for example for axis-parallel
ellipses or hyperbolas, we get the bound $O\left(m^{1/2}n^{7/8+\eps} + m^{2/3}n^{2/3} + m + n \right)$, the best previously known bound for this case was the Pach-Sharir bound
$O(m^{4/7}n^{6/7}+m+n)$. The new bound is superior when $m\ge n^{1/4+c}$ for any constant $c>0$ (with $\eps$ depending on $c$, as above).
\end{remark}

The high-level approach that we use follows the earlier treatments that have appeared, for example, in Agarwal et al.~\cite{ANPPSS}. 
We first derive a weaker bound using Sz\'ekely's crossing lemma argument for collections of pseudo-segments~\cite{Sze}, and then strengthen this bound
by passing to a parametric dual space in which the curves of $\mathcal{C}$ become points, and the points of $\P$ become 
bounded-degree algebraic hypersurfaces. We then decompose the problem into smaller subproblems, using the multilevel 
polynomial partitioning technique of Matou\v{s}ek and Pat\'akov\'a~\cite{MP} (see Theorem~\ref{thm:mp} below). Finally, we use induction (or 
rather recursion) on the subproblems produced by the partition. The terminal instances of the recursion are subproblems 
which are either too small, or at which we can effectively apply the weak bound.

\subsection{An initial weaker bound}
\begin{lemma}\label{STResultLem}
Let $\P$ be a set of $m$ points and let $\C$ be a set of $n$ plane algebraic curves of degree at most $D$, no two of which share a common component. Then 
\begin{equation}\label{boundOnIPC}
 I(\P,\C)= O_D(m^{2/3}n^{2/3}+n^{3/2}\log^{O_D(1)}n+m).
\end{equation}
\end{lemma}
\begin{proof}
Apply Theorem \ref{cuttingCurvesIntoSegments} to $\C$, and let $\Gamma_0$ be the resulting set of Jordan arcs. If $(p,C)$ is an incidence from $I(\P,\C)$, then either (a) $p$ is a singular
point of $C$, or (b) there is a curve $\gamma\in\Gamma_0$ with $\gamma\subset C$ and either (b.i) $p\in\gamma$ or (b.ii) $p$ lies at an endpoint
of $\gamma$ (recall that the curves in $\Gamma_0$ are relatively open). We conclude that

\begin{equation*}
 I(\P,\C)\leq I(\P,\Gamma_0)+2|\Gamma_0|+D^2|\C|.
\end{equation*}
The bound in (\ref{boundOnIPC}) then follows by
applying the Szem\'eredi-Trotter theorem for incidences with pseudo-segments, using the crossing-lemma technique of Sz\'ekely~\cite{Sze}, the bound
$|\Gamma_0| = O_D(n^{3/2}\log^{O_D(1)}n)$ from Theorem \ref{cuttingCurvesIntoSegments}, and the fact that the number of crossings
between the arcs of $\Gamma_0$ is still only $O(n^2)$.
\end{proof}

\begin{remark}
 In fact, the above argument actually proves a slightly stronger statement. If the set of algebraic curves $\mathcal{C}$ is replaced by a set of
Jordan arcs $\Gamma$ that satisfy the hypotheses of Theorem \ref{cuttingArcsIntoSegments} (or if we still stick to full algebraic curves
with a smaller number of intersections), then \eqref{boundOnIPC} can be replaced by the stronger bound 
 \begin{equation*}
 I(P,\Gamma) = O_D\left( m^{2/3}X^{1/3} + m + n + n^{1/2}X^{1/2}\log^{O_D(1)} n \right),
\end{equation*}
 where $X=\sum_{\substack{\gamma,\gamma^\prime\in\Gamma\\ \gamma\neq\gamma^\prime}}|\gamma\cap\gamma^\prime|$. 
 The last term follows from the refined bound on $|\Gamma_0|$ given in Theorem~\ref{cuttingArcsIntoSegments}, 
 and the first term follows from Sz\'ekely's crossing-lemma analysis~\cite{Sze}.
\end{remark}
\subsection{Duality and space decomposition}\label{dualityTransSec}
In this section we will describe a ``duality transform'' that sends algebraic curves to points in a suitable parameter space, and sends points in the plane to algebraic varieties in this parameter space. The key property of this transform is that it preserves the incidence relation---if a curve in the plane is incident to a point, then the corresponding point and variety in the parameter space are incident too. 

In the statement of Theorem \ref{incidencesPtsCurves}, we refer to a family of algebraic curves of degree at most $D$, which, by definition, is a subvariety $F\subset \mathbf{P}\RR^{\binom{D+2}{2}}$. $F$
need not be irreducible, so let $F^\prime\subset F$ be an irreducible component of $F$; we will consider each irreducible component of $F$ separately. If $F^\prime=
\mathbf{P}\RR^{\binom{D+2}{2}}$, then the set of curves that are incident to a point $p$ in the plane corresponds to a proper subvariety $\sigma_p$ of $F^\prime$
(in fact, $\sigma_p$ is a hyperplane). However, if $F^\prime$ is a
proper subvariety of $\mathbf{P}\RR^{\binom{D+2}{2}}$, then it is possible that there is a point $p$ that is incident to every curve in $F^\prime$. This can occur, but if it does, then either
there are $\leq D^2$ points $p\in\RR^2$ with this property, or $F^\prime$ can contain at most one curve from $\C$. Thus if we throw away a small set of points and curves, we can assume that for each point $p$
in the plane, the set of curves that are incident to $p$ corresponds to a (proper) subvariety of $F^\prime$ (this set is in fact the intersection of $F'$ with some hyperplane). The number of incidences that we may have missed is at most $O_{D,\operatorname{deg}(F)}(n+m)$, and we can deal with these incidences separately.
The following lemma makes this statement precise.

\begin{lemma}[point-curve duality]\label{dualityTransformLem}
 Let $F\subset \mathbf{P}\mathbb{R}^{\binom{D+2}{2}}$ be a family of plane curves of degree at most $D$ with $\dim F=s$. Let $\mathcal{C}\subset F$ be a
finite set of curves, no pair of which share a common component, and let $\P\subset\RR^2$ be a finite set of points. Then there exist a set of points $\P_{\operatorname{bad}}\subset\RR^2$, a set of curves $\C_{\operatorname{bad}}\subset \C$, a set of points $W=\{w_C\}_{C\in\mathcal{C}\backslash \C_{\operatorname{bad}}}\subset\RR^s$ and a set of real algebraic varieties $\Sigma=\{\sigma_p\}_{p\in\P\backslash \P_{\operatorname{bad}}}$ in $\RR^s$ that satisfy the following properties:
 \begin{itemize}
  \item Each variety $\sigma_p$ has dimension at most $s-1$ and degree $O_{D,\deg(F)}(1)$.
  \item If $C\in\mathcal{C}\backslash \C_{\operatorname{bad}}$ and $p\in \P\backslash\P_{\operatorname{bad}}$, then $p\in C$ if and only if $w_C\in \sigma_p$.
  \item $|\P_{\operatorname{bad}}|=O_{D,\deg(F)}(1)$ and $|\C_{\operatorname{bad}}|=O_{D,\deg(F)}(1)$.
 \end{itemize}
\end{lemma}

\begin{proof}
Decompose $F$ into its irreducible components $F_1\cup\cdots\cup F_\ell$. If an irreducible component contains at most one curve from $\C$, add this curve to $\C_{\operatorname{bad}};$ after
doing so, $|\C_{\operatorname{bad}}|=O_{D,\deg(F)}(1)$. After re-indexing, we will assume that each of the remaining components $F_1,\ldots,F_{\ell^\prime}$ contain at least two curves from $\C$. Let $F^\prime=F_1\cup\cdots\cup F_{\ell^\prime}$.

For each $p\in\RR^2$, define $H_p=\{\gamma\in \mathbf{P}\RR^{\binom{D+2}{2}} \mid p\in\gamma\}$. Then $H_p$ is a subvariety of $\mathbf{P}\RR^{\binom{D+2}{2}}$ (in fact, it it a hyperplane).
Observe that if $F_j$ contains at least two curves from $\C$, then $F_j\subset H_p$ for at most $D^2$ points $p\in\RR^2$. Indeed, suppose there exist points $p_1,\ldots,p_{D^2+1}$ with
$F_j\subset H_{p_i}$ for each $i=1,\ldots,D^2+1.$ Let $C_1,C_2$ be distinct curves in $\C\cap F_j$. Then $C_1$ and $C_2$ intersect in $\geq D^2+1$ points, so by B\'ezout's theorem, they must
share a common component, contrary to our assumptions. Define 
\begin{equation*}
\P_{\operatorname{bad}}=\bigcup_{j=1}^{\ell^\prime}\{p\in\RR^2 \mid F_j\subset H_p\}.
\end{equation*}
We have $|\P_{\operatorname{bad}}|=O_{D,\deg(F)}(1)$. 

Now, if $p\in\RR^2\backslash \P_{\operatorname{bad}}$, then $H_p\cap F_j$ is a proper subvariety of $F_j$ for each $j=1,\ldots,\ell^\prime$. Thus $H_p\cap F^\prime$ is a proper subvariety of $F^\prime$ of degree $O_{D,\operatorname{deg}(F)}(1)$. 

For the next step, we need to identify (a Zariski open subset of) $\mathbf{P}\mathbb{R}^{\binom{D+2}{2}}$ with $\RR^{\binom{D+2}{2}-1}$. To do this, we need to choose
where the ``hyperplane at infinity'' lies. We wish to do this in a way that does not affect the incidence relation between the points (representing curves) of $\mathcal{C}$ and the surfaces $\{H_p\}$. Let $H\subset\mathbf{P}\mathbb{R}^{\binom{D+2}{2}}$ be a generic hyperplane (in particular, $H$ avoids all the points of $\mathcal{C}$, and $H\neq H_p$ for any $p\in\P$). After a change of coordinates, we can assume that $H$ is the hyperplane $\{x_0=0\}\subset \mathbf{P}\mathbb{R}^{\binom{D+2}{2}}.$ With this choice of $H$, we obtain a function 
\begin{equation*}
\begin{split}
\operatorname{Aff} : \; &\mathbf{P}\mathbb{R}^{\binom{D+2}{2}}\backslash H\to \RR^{\binom{D+2}{2}-1}, \quad\text{defined by}\\
&[x_0:x_1:\ldots:x_{\binom{D+2}{2}-1}]\mapsto(x_1/x_0,\ldots,x_{\binom{D+2}{2}-1}/x_0).
\end{split}
\end{equation*}

Let $\pi :\; \RR^{\binom{D+2}{2}-1}\to\RR^s$ be a generic surjective linear transformation (i.e., $\pi$ is given by a
generic\footnote{Over $\RR$ one must be a bit careful with ``generic'' points, since they lack some of the favorable properties that hold over
an algebraically closed field. See \cite[\S4.2]{Zahl} for further discussion of generic points in the context of combinatorial geometry.}
$\left(\binom{D+2}{2}-1\right)\times s$ matrix, which will necessarily have rank $s$). 

For each $p\in\P\backslash\P_0$, define $\sigma_p$ to be the Zariski closure of $\pi(\operatorname{Aff}(F^\prime\cap H_p))$ (since $\pi$ and $H$ were chosen
generically, $\pi(\operatorname{Aff}(H_p))$ is already a real variety, but it is easier to take the Zariski closure than to verify this fact).
We have $\sigma_p\subset \mathbb{R}^s$, and $\sigma_p$  is a real algebraic variety of dimension at most $s-1$ and degree $O_{D,\deg(F)}(1)$.

For each $C\in\C\backslash\C_{\operatorname{bad}}$, define $w_C = \pi(\operatorname{Aff}(C))$; this is a point in $\RR^s$. Define $W=\{w_C\}_{C\in\C\backslash\C_{\operatorname{bad}}}$. Since
$\pi$ was chosen generically, it preserves the incidence relation between the points $C\in\mathcal{C}$ and the surfaces $H_p$. Thus if $p\in\P\backslash\P_{\operatorname{bad}}$ and $C\in\C\backslash\C_{\operatorname{bad}}$, then $p\in C$ if and only if $w_C\in \sigma_p$. 
\end{proof}
\begin{remark}
The objects created by Lemma \ref{dualityTransformLem} appear rather suspicious. We know that in dimensions $\geq 3$, it is impossible to get non-trivial point-hypersurface incidence theorems unless we impose some sort of non-degeneracy condition on the points and surfaces. Otherwise, it is possible that all of the hypersurfaces intersect in a common curve (or higher dimensional variety), and all of the points lie on this curve. On the face of it, we have not ruled out this possibility, so it seems strange that we will be able to use Lemma \ref{dualityTransformLem} to obtain non-trivial incidence theorems. 

However, since the points and varieties produced by Lemma \ref{dualityTransformLem} come from collections of points and curves in $\RR^2$, we will be able to exclude the sort of degenerate
arrangements that prevent non-trivial incidence results. This will be made explicit in the bound \eqref{FPSSZBound} below, which exploits the fact that the incidence graph of points and curves cannot contain a large induced bipartite subgraph.
\end{remark}

\subsection{Multi-level polynomial partitioning}
The duality transform from Lemma \ref{dualityTransformLem} allows us to recast our incidence problem involving points and curves in the plane as a new incidence problem involving points and varieties in $\RR^s$. We will analyze this new problem using the following multilevel polynomial partitioning theorem of Matou\v{s}ek and Pat\'akov\'a~\cite{MP}, which generalizes the polynomial partitioning theorem of Guth and Katz from \cite{GK2}.

\begin{theorem}[Matou\v{s}ek and Pat\'akov\'a~\protect{\cite[Theorem 1.1]{MP}}] \label{thm:mp}
For every integer $s > 1$ there is a constant $K$ such that the following holds.
Given a set $\P \subset\reals^s$ of cardinality $n$ and a parameter $r > 1$, there are 
numbers $r_1, r_2,\ldots, r_s \in [r, r^K]$, positive integers $t_1, t_2, \ldots, t_s$, 
a partition
\begin{equation*}
\P = \P^* \cup \bigcup_{i=1}^s \bigcup_{j=1}^{t_i} \P_{ij}
\end{equation*}
of $\P$ into pairwise disjoint subsets, and for every $i, j$, a connected set $S_{ij}\subseteq\reals^s$
containing $\P_{ij}$, such that $|\P_{ij}| \le n/r_i$ for all $i, j$; $|\P^*| \le r^K$; and the 
following holds:

Let $Z\subset\RR^s$ be a variety of degree at most $D$. Then for each $i = 1,2,\ldots,s$, the number of sets $S_{ij}$ that cross $Z$ is $O_{D,d}\left(r_i^{1-1/d}\right)$.
\end{theorem}
\noindent (In the theorem, $Z$ \emph{crosses} $S_{ij}$ if $Z\cap S_{ij} \ne\emptyset$ 
but $Z$ does not contain $S_{ij}$.)

While it is not stated explicitly in \cite{MP}, we also have the bound
\begin{equation}\label{boundOnNumberOfPieces}
 \sum_{i=1}^s t_i\leq r^K,
\end{equation}
provided $K$ is chosen sufficiently large (depending only on $s$). In brief, the bound \eqref{boundOnNumberOfPieces} is obtained as follows. The proof of \cite[Theorem 1.1]{MP} constructs a
sequence of $s$ polynomials, each of degree at most $r^{K^\prime}$ (where $K^\prime$ depends only on $s$), and the sets $S_{ij}$ are the connected components of all realizable sign conditions of these polynomials. By \cite{BPR}, the set of connected components of sign conditions determined by $s$ polynomials in $\RR^s$, each of degree at most $r^{K^\prime}$, has cardinality at most 
$O_s(1) r^{sK^\prime}$. Thus if $r>1$ (which is assumed to be the case) and $K$ is chosen sufficiently large, then \eqref{boundOnNumberOfPieces} holds.

\subsection{Proof of Theorem \ref{incidencesPtsCurves}}
We are now ready to prove Theorem \ref{incidencesPtsCurves}. First, note that \eqref{newIncidenceBdPrecise} immediately holds if $m \ge n^{5/4+\eps'}$, where $\eps'$ is a suitable multiple of $\eps$ (see below for a concrete choice). 
Indeed, we then have 
\begin{equation*}
n^{3/2}\polylog n = O_D\left(n^{3/2+2\eps'/3}\right) = O_D\left( m^{2/3}n^{2/3} \right) ,
\end{equation*}
so using Lemma \ref{STResultLem}, we obtain
\begin{equation*}
I(\P,\Gamma) = O_D\left( m^{2/3}n^{2/3} + m \right).
\end{equation*}
Henceforth we will assume that $m < n^{5/4+\eps'}$.

The proof for this case proceeds by induction on $m$ and $n$.
Concretely, given $\eps$, $D$, $s$, and $F$, we establish the bound
\begin{equation} \label{indbd}
I(\P,\C) \le A
m^{\frac{2s}{5s-4}} n^{\frac{5s-6}{5s-4}+\eps} + B\left(m + n \right) , 
\end{equation}
for any sets $\P\subset\RR^2$ with $|\P|=m$, $\C\subset F$, $|\C|=n$, and $m < n^{5/4+\eps'}$, where $A = O_{\eps,D,s,\deg(F)}(1)$ and $B=O_{D,s,\deg(F)}(1)$.

The induction, or rather recursion, bottoms out in three cases:
\begin{enumerate}
\item[(i)] We reach a subproblem with fewer than $r$ points, for some suitable constant
parameter $r$ whose value will be set later.
\item[(ii)] We reach a subproblem with $m \ge n^{5/4+\eps'}$
\item[(iii)] We reach a subproblem with $m \le n^{1/s}$.
\end{enumerate}

In all three cases, \eqref{indbd} holds provided we choose $A$ sufficiently large.
This is clear for case (i), and requires some justification for case (ii) (provided below).
For case (iii), we use the fact that the incidence subgraph of $\P\times\C$ is a semi-algebraic graph in $\RR^2\times\RR^s$, and this graph does not contain a large complete bipartite subgraph. 
Indeed, by B\'ezout's theorem, for example, it does not contain a copy of $K_{D^2+1,2}$ as a subgraph. By Corollary 2.3 from Fox et al.~\cite{FPSSZ}, this implies that 
\begin{equation}\label{FPSSZBound}
I(\P,\C) = O_{s,D,\deg(F)}\left( mn^{1-1/s} + n \right).
\end{equation}
If $m\le n^{1/s}$ then this quantity is $O_{s,D,\deg(F)}(n)$. Thus if we choose $A$ sufficiently large, the bound in \eqref{indbd} holds in this case.

Apply Lemma \ref{dualityTransformLem} to $\P$ and $\C$. Let $W\subset\RR^s$ be the resulting set of points, let $\Sigma$ be the resulting set of varieties, and let $\P_{\operatorname{bad}},\
\C_{\operatorname{bad}}$ be the leftover sets of problematic points and curves.

Apply Theorem~\ref{thm:mp} to $W$ with a value of $r$ that will be specified later. Let $K$, $r_1,\ldots,r_s$, $t_1,\ldots,t_s$ be the parameters given by the theorem; let
\begin{equation*}
W = W^* \cup \bigcup_{i=1}^s \bigcup_{j=1}^{t_i} W_{ij}
\end{equation*}
be the corresponding partition of $W$; and for each index $i$ and $j$, let $S_{ij}$ be the connected set that contains 
$W_{ij}$. We have
\begin{equation}\label{notBadIncidences}
I(\P\backslash \P_{\operatorname{bad}},\C\backslash \C_{\operatorname{bad}}) = I(W,\Sigma)  = I(W^*,\Sigma) +
\sum_{i=1}^s \sum_{j=1}^{t_i} \Biggl( I(W_{ij},\Sigma_{ij}) + I(W_{ij},\Sigma^0_{ij}) \Biggr),
\end{equation}
where $\Sigma_{ij}$ (resp., $\Sigma^0_{ij}$) is the set of the surfaces of $\Sigma$
that cross (resp., contain) the corresponding set $S_{ij}$. Since $|\P_{\operatorname{bad}}|=O_{D,\deg(F)}(1)$ and $|\C_{\operatorname{bad}}|=O_{D,\deg(F)}(1)$, we have
\begin{equation} \label{pbadc}
I(\\P_{\operatorname{bad}},\C)=O_{D,\deg(F)}(n) ,
\end{equation}
and 
\begin{equation} \label{pcbad}
I(\P,\C_{\operatorname{bad}})=O_{D,\deg(F)}(m).
\end{equation}
Thus it suffices to bound the contribution from \eqref{notBadIncidences}.

Let $m_{ij} = |\Sigma_{ij}|$, $m^0_{ij} = |\Sigma^0_{ij}|$, and $n_{ij} = |W_{ij}|$,
for each $i$, $j$. We have $n_{ij} \le n/{r_i}$ for each $i,j$, and
\begin{equation*}
\sum_{j=1}^{t_i} m_{ij} \le bmr_i^{1-1/s} ,
\end{equation*}
for each $i$, where $b$ is a constant that depends on $s$ and $D$.

\paragraph{Incidences with crossing surfaces.}
We apply the induction hypothesis to each $I(W_{ij},\Sigma_{ij})$ for which
$m_{ij} = |\Sigma_{ij}| < |W_{ij}|^{5/4+\eps'} = n_{ij}^{5/4+\eps'}$.
For the remaining indices $i$, $j$, where $m_{ij} \ge n_{ij}^{5/4+\eps'}$,
we use the fact that $m < n^{5/4+\eps'}$ (or else we would not have applied
the partitioning to $W$ and $\Sigma$). We can verify that
\begin{equation*}
m^{2/3}n^{2/3} \le m^{\frac{2s}{5s-4}} n^{\frac{5s-6}{5s-4}+\eps} 
\end{equation*}
if and only if
\begin{equation*}
m \le n^{\frac54 + \frac{3\eps(5s-4)}{4s-8}} ,
\end{equation*}
and the latter inequality holds if we ensure that $\eps' \le \frac{3\eps(5s-4)}{4s-8}$.

On the other hand, we have
\begin{equation*}
n^{3/2} \le m^{\frac{2s}{5s-4}} n^{\frac{5s-6}{5s-4}+\eps_1} 
\end{equation*}
if and only if
\begin{equation*}
m \ge n^{\frac54 - \frac{\eps_1(5s-4)}{2s}} .
\end{equation*}
The latter inequality holds for $m_{ij}$ and $n_{ij}$, for any value of $\eps_1$, by assumption.
Hence, when we reach a subproblem of this kind, we use the weak bound from  Lemma \ref{STResultLem}, and get
\begin{align*}
I(W_{ij},\Sigma_{ij}) & = O_{s,D,\deg(F)}\left( m_{ij}^{2/3}n_{ij}^{2/3} + n_{ij}^{3/2} \polylog n \right) \\
& = O_{s,D,\deg(F)}\left( m^{2/3}n^{2/3} + n^{3/2} \polylog n \right) \\
&=O_{s,D,\deg(F)}\left( m^{\frac{2s}{5s-4}} n^{\frac{5s-6}{5s-4}+\eps} \right) .
\end{align*}
Summing this bound over all relevant $i$ and $j$ multiplies the bound by a constant factor
that depends on $s$, $D$, and $r$, so the overall contribution to the incidence count 
by ``borderline'' subproblems of this kind is at most
\begin{equation*}
B_1m^{\frac{2s}{5s-4}} n^{\frac{5s-6}{5s-4}+\eps} ,
\end{equation*}
for a suitable constant $B_1$ that depends on $s$, $D$, $\deg(F)$, and $r$.

For subproblems satisfying $m_{ij} < n_{ij}^{5/4+\eps'}$, we get from the induction hypothesis
\begin{equation*}
I(W_{ij},\Sigma_{ij}) \le A
m_{ij}^{\frac{2s}{5s-4}} n_{ij}^{\frac{5s-6}{5s-4}+\eps} + B\left(m_{ij} + n_{ij} \right) .
\end{equation*}
Therefore, for each fixed $i$,
\begin{equation*}
\sum_{j=1}^{t_i} I(W_{ij},\Sigma_{ij}) \le  
\sum_{j=1}^{t_i} \left(A
m_{ij}^{\frac{2s}{5s-4}} n_{ij}^{\frac{5s-6}{5s-4}+\eps} + B(m_{ij} + n_{ij}) \right) .
\end{equation*}
We have 
\begin{equation} \label{mandn}
\sum_{j=1}^{t_i} m_{ij} \le bmr_i^{1-1/s} ,\quad\quad\text{and}\quad\quad
\sum_{j=1}^{t_i} n_{ij} = |W_i| ,
\end{equation}
where $W_i = \bigcup_{j=1}^{t_i} W_{ij}$. Using H\"older's inequality to bound the sum of the first terms, we obtain
\begin{align} \label{bd:cross}
\sum_{j=1}^{t_i} Am_{ij}^{\frac{2s}{5s-4}} n_{ij}^{\frac{5s-6}{5s-4}+\eps} 
& \le A\sum_{j=1}^{t_i} m_{ij}^{\frac{2s}{5s-4}} n_{ij}^{\frac{3s-4}{5s-4}}
\left( \frac{n}{r_i} \right)^{\frac{2s-2}{5s-4}+\eps} \nonumber \\
& \le A\left( \sum_{j=1}^{t_i} m_{ij} \right)^{\frac{2s}{5s-4}} 
\left( \sum_{j=1}^{t_i} n_{ij} \right)^{\frac{3s-4}{5s-4}}
\left( \frac{n}{r_i} \right)^{\frac{2s-2}{5s-4}+\eps} \nonumber \\
& \le A \left( bmr_i^{1-1/s} \right)^{\frac{2s}{5s-4}} 
|W_i|^{\frac{3s-4}{5s-4}}
\left( \frac{n}{r_i} \right)^{\frac{2s-2}{5s-4}+\eps} \nonumber \\
& = A\frac{b'}{r_i^\eps} 
m^{\frac{2s}{5s-4}} |W_i|^{\frac{3s-4}{5s-4}} n^{\frac{2s-2}{5s-4}+\eps} 
\quad\quad\text{(for $b' = b^{\frac{2s}{5s-4}}$)} \nonumber \\
& \le A\frac{b'}{r^\eps} 
m^{\frac{2s}{5s-4}} |W_i|^{\frac{3s-4}{5s-4}} n^{\frac{2s-2}{5s-4}+\eps} ,
\end{align}
recalling that $r_i\ge r$ for each $i$. 
We now sum these bounds over all $i=1,\ldots,s$, and get a total of at most
\begin{equation*}
\frac{Ab's}{r^\eps} 
m^{\frac{2s}{5s-4}} n^{\frac{5s-6}{5s-4}+\eps} .
\end{equation*}

In total, the number of incidences involving crossing surfaces is at most
\begin{equation} \label{xinc}
\left( \frac{Ab's}{r^\eps} + B_1\right)
m^{\frac{2s}{5s-4}} n^{\frac{5s-6}{5s-4}+\eps} +
\left( Bb\sum_{i=1}^s r_i^{1-1/s} \right) m + Bn .
\end{equation}

\paragraph{Incidences with containing surfaces.}
Fix $i$ and $j$, and consider the incidence count $I(W_{ij},\Sigma^0_{ij})$.
All the points of $W_{ij}$ lie in the corresponding containing set $S_{ij}$,
and all the surfaces of $\Sigma^0_{ij}$ contain $S_{ij}$. Consequently,
every pair in $W_{ij} \times \Sigma^0_{ij}$ is an incident pair. However, by
assumption, the incidence graph between $W_{ij}$ and $\Sigma^0_{ij}$ does not 
contain $K_{2,D^2+1}$. This implies that
\begin{equation*}
I(W_{ij},\Sigma^0_{ij}) \le D^2|W_{ij}| + |\Sigma^0_{ij}| = D^2n_{ij} + m^0_{ij} 
\end{equation*}
(the first (resp., second) term accounts for sets $W_{ij}$ of size at least two (resp., at most one).
Hence, summing these bounds over all $i,j$, using
the trivial bound $m^0_{ij}\le m$, for all $i$, $j$, and the bound $\sum_{i,j} n_{ij} \le n$, we get
\begin{equation} \label{inc-cont}
\sum_{i=1}^s \sum_{j=1}^{t_i} I(W_{ij},\Sigma^0_{ij}) \le 
\sum_{i=1}^s \sum_{j=1}^{t_i} \left( D^2 n_{ij} + m^0_{ij} \right) \le B_2m + D^2 n ,
\end{equation}
where $B_2$ is another constant that depends on $r$, $s$ and $D$.
(It is here that we use the remark following Theorem~\ref{thm:mp}, concerning a bound on the quantities $t_i$.)

Finally, we bound $I(W^*,\Sigma)$ simply by $mr^K$.
Adding all bounds collected so far, in \eqref{pbadc}, \eqref{pcbad}, \eqref{xinc}, and \eqref{inc-cont}, we get a total of at most
\begin{equation*}
\left( \frac{Ab's}{r^\eps} + B_1\right)
m^{\frac{2s}{5s-4}} n^{\frac{5s-6}{5s-4}+\eps} +
B_3 m + B_4 n ,
\end{equation*}
where $B_3$ and $B_4$ are constants that depend on ($A$, $B$, and) $s$, $D$, $\deg(F)$, and $r$. We now observe that
\begin{equation*}
m \le m^{\frac{2s}{5s-4}} n^{\frac{5s-6}{5s-4}} \quad\text{if and only if}\quad
m \le n^{\frac{5s-6}{3s-4}} ,
\end{equation*}
In our case we have the stronger inequality $m < n^{5/4+\eps'}$; it is indeed stronger for $\eps'<1/4$, say,
as can easily be verified. We thus have
\begin{equation*}
B_3 m \le \frac{B_3}{n^\eps} 
m^{\frac{2s}{5s-4}} n^{\frac{5s-6}{5s-4}+\eps}.
\end{equation*}
Similarly, we have
\begin{equation*}
n \le m^{\frac{2s}{5s-4}} n^{\frac{5s-6}{5s-4}} \quad\text{if and only if}\quad
m \ge n^{1/s} ,
\end{equation*}
which also holds by our recursion termination rules. Hence we have
\begin{equation*}
B_4 n \le
\frac{B_4}{n^\eps} 
m^{\frac{2s}{5s-4}} n^{\frac{5s-6}{5s-4}+\eps} .
\end{equation*}
Altogether, the incidence bound is at most 
\begin{equation*}
\left( \frac{Ab's}{r^\eps} + B_1 + \frac{B_3+B_4}{n^\eps} \right)
m^{\frac{2s}{5s-4}} n^{\frac{5s-6}{5s-4}+\eps} .
\end{equation*}
We now take $r$ to be sufficiently large, so as to have $r^\eps > 3b's$ (recalling that $b'$ does not depend on $r$), take $A$ 
sufficiently large so that $B_1 < A/3$, and then require $n$ to be sufficiently large
so that
\begin{equation*}
\frac{B_3+B_4}{n^\eps} < \frac{A}{3} .
\end{equation*}
With these choices, this expression is upper bounded by
\begin{equation*}
A m^{\frac{2s}{5s-4}} n^{\frac{5s-6}{5s-4}+\eps} ,
\end{equation*}
which establishes the induction step and thereby completes the proof of the theorem.
$\Box$

\section{The complexity of a level in an arrangement of curves}
\label{subsec:level}
Recall the definition of a level in an arrangement of curves from Section \ref{complexityOfALevelSec}.
The main tool for establishing bounds on the complexity of levels in
arrangements of curves is an upper bound given by 
Chan~\cite{Ch} on the complexity of a level in an
arrangement of extendible pseudo-segments. 

A collection of $x$-monotone Jordan arcs is \emph{extendible} if each arc can be contained in a $x$-monotone simple curve that divides the plane into exactly two connected components, with the property that these larger curves form a
collection of pseudo-lines (a collection of curves is called a collection of pseudo-lines if the curves are unbounded and every pair of curves intersect at most once). Chan established the following bound on the complexity of a level of an arrangement of extendible pseudo-segments.

\begin{theorem}[Chan~\cite{Ch}, Theorem 2.1]\label{chanThm1}
Let $\Gamma$ be a collection of $n$ extendible pseudo-segments, and let $X=\sum_{\gamma,\gamma^\prime\in\Gamma}|\gamma\cap\gamma^\prime|$. The complexity of a level in $A(\Gamma$) is $O(n+n^{2/3}X^{1/3})$.
\end{theorem}

In general, a collection of pseudo-segments need not be extendible. However, any collection of pseudo-segments can be cut into a slightly larger collection that is extendible.
\begin{theorem}[Chan~\cite{Ch}, Theorem 3.3]\label{chanThm2}
Any collection of $n$ $x$-monotone pseudo-segments can be cut into a collection of $O(n\log n)$ extendible pseudo-segments.
\end{theorem}

Combining Theorems \ref{chanThm1} and \ref{chanThm2} with the bounds in Theorems \ref{cuttingCurvesIntoSegments} and \ref{cuttingArcsIntoSegments}, we obtain the following result.

\begin{levelsInArrThm}
Let $\Gamma$ be a set of $n$ Jordan arcs, each of which is contained in an algebraic curve of degree at most $D$, and every pair of which have finite intersection. Then each level of $\A(\Gamma)$ has complexity
$O_D(n^{5/3}\log^{O_D(1)} n)$. 
\end{levelsInArrThm}

This result improves earlier works of Chan~\cite{Ch,Ch2} and Bien~\cite{lilach}
for the case of general algebraic curves, and it almost matches the earlier results
in \cite{ANPPSS,MT} for the case of pseudo-circles and pseudo-parabolas.

\begin{remark}
It is an interesting open problem to obtain a refined bound on the complexity
of the $k$-level which depends on $k$. Such a bound is known for the case of lines
(and pseudo-lines)~\cite{Dey}.
\end{remark}

As noted in \cite{ANPPSS}, the preceding theorem implies the following 
result in the area of kinetic geometry. This significantly extends
the earlier results in \cite{ANPPSS,TT}, which were limited to the case of
constant-velocity motions. 

\begin{corollary} \label{median}
Let $P$ be a set of $n$ points in the plane, each
moving along some algebraic trajectory of degree at most $D$(the coordinates
of the position of a point at time $t$ are polynomials of degree at most $D$). For each time $t$, let $p(t)$ and $q(t)$ be the pair of 
points of $P$ whose distance is the median distance at time $t$. 
The number of times in which this median pair changes is $O_D(n^{10/3}\log^{O_D(1)} n)$. The same bound applies if the median is replaced by any fixed quantile.
\end{corollary}

\section{The complexity of many marked faces in an arrangement}\label{sec:manyf}

In this section we prove Theorem \ref{manyfacesCurvesWeakBd}. 
Recall the setup from Section \ref{complexityMarkedFacesIntroSec}: 
Let $\Gamma$ be a set of $n$ Jordan arcs, each of pair of which have finite intersection. Let $\P$ be a set of $m$ points in the plane with the property that no point of $\P$ lies on any curve of $\Gamma$. We define $K(\P,\Gamma)$ to be the sum of the complexities of the faces of $\A(\Gamma)$ that contain at least one point of $\P$, where the complexity of a face is the number of edges of $\A(\Gamma)$ 
on its boundary.

For the reader's convenience, we restate the theorem here.

\begin{manyfacesCurvesWeakBdThm}
Let $\C$ be a set of algebraic plane curves of degree at most $D$, no two of which share a common component. Let $\Gamma$ be a set of $n$ Jordan arcs, each of which is contained in some curve of $\C$, and each pair of which have finite intersection. Let $\P$ be a set of $m$ points in the plane, so that no point of $\P$ lies on any curve of $\C$. Then 
\begin{equation}\tag{\ref{weakptfaces}}
K(\P,\Gamma) = O_D(m^{2/3}n^{2/3}+n^{3/2}\log^{O_D(1)}n).
\end{equation} 
\end{manyfacesCurvesWeakBdThm}

\begin{proof}
The bound is an immediate consequence of the results in~\cite{AAS}, combined with
Theorems \ref{cuttingCurvesIntoSegments} and \ref{cuttingArcsIntoSegments}. Specifically, Theorem 3.5
in \cite{AAS} asserts that the complexity of $m$ marked faces in an arrangement of $N$ pseudo-segments with $X$ intersection points is
$O(m^{2/3}X^{1/3} + N\log^2 N)$. Applying this bound to the collection of pseudo-segments produced in
Theorem~\ref{cuttingCurvesIntoSegments} or Theorem~\ref{cuttingArcsIntoSegments} yields the bound stated in \eqref{weakptfaces}.

We note that the bound \eqref{weakptfaces} parallels the weak incidence bound in \eqref{boundOnIPC},
except for the missing term $O(m)$ and the fact that the exponent in
the polylogarithmic factor is now larger by $2$. 
We also note that the term $O(N\log^2N)$ reduces to $O(N\log N)$ when the pseudo-segments are extendible; the extra logarithmic factor comes 
from Theorem~\ref{chanThm2}.
\end{proof}

\subsection{Discussion}\label{markedFacesDiscussionSec}
As in the case of incidences, one would like to improve Theorem \ref{manyfacesCurvesWeakBd} and obtain a refined bound, similar to that in
Theorem~\ref{incidencesPtsCurves}. However, the case of many faces is considerably more difficult, and it raises several
technical issues that, so far, we do not know how to overcome. We briefly discuss these difficulties, and leave this extension as
an interesting open problem.

The approach, as in the case of incidences, would be to pass to the dual $s$-dimensional space. In the dual space, curves become points and the marking points become algebraic varieties. We even have a slight advantage here, because we can perturb the marking points slightly to ensure that they are in general position. One would then apply a polynomial partitioning in the dual space, apply the bound of
Theorem \ref{manyfacesCurvesWeakBd} within each cell, and combine the bounds into a global bound for the whole problem.
However, there are several major issues that arise here.

\medskip

\noindent{\bf (a)}
Within a cell $\tau$ of the partition, we have a subset $\P_\tau$ of points of $\P$ whose dual surfaces cross $\tau$, and a subset $\Gamma_\tau$
of curves of $\Gamma$ whose dual points lie in $\tau$. The recursive subproblem at $\tau$ would then be to bound the complexity of the faces
marked by the points of $\P_\tau$ in the arrangement $\A(\Gamma_\tau)$. However, this is not enough, as the points of
$\P\backslash\P_\tau$ also mark faces of $\A(\Gamma_\tau)$, and we have to estimate the complexity of these faces as well. Informally, this is an
effect of the ``non-local'' nature of the curve-face incidence relation: in contrast to the case of point-curve incidences, the property that a curve $\gamma$ bounds the face marked by a point $p$ is a global property that depends on the whole collection of curves and not just on $p$ and $\gamma$. In general, the complexity of these (many) additional faces of $\A(\Gamma_\tau)$ could be too large for the recursive analysis to
yield the desired improved bound. 

\medskip

\noindent{\bf (b)}
As in the case of incidences, we need to bootstrap the recursion at subproblems for which $|\P_\tau|$ is much smaller than $|\Gamma_\tau|$.
Concretely, if we are to obtain the same bound as for incidences, the threshold would be
$|\P_\tau| \le |\Gamma_\tau|^{1/s}$. We would then need to argue that in this case the complexity of the marked faces is linear, or at least
close to linear, in $|\Gamma_\tau|$. Again, the non-local nature of the problem makes it sifficult to show this. For example, we do not know
whether the machinery in Fox et al.~\cite{FPSSZ} can be applied here, as it was in the case of incidences.

\medskip

\noindent{\bf (c)}
When combining the bounds obtained at the recursive subproblems into a global bound, there are several additional technical issues that are more challenging when bounding the complexity of marked faces rather than incidences. For example, unlike the case of incidences, we cannot just add up the recursive bounds. This is because the structure of faces in an arrangement obtained by overlaying several sub-arrangements can become quite involved. Fortunately, the techniques in Agarwal et al.~\cite{AAS} provide a solution to this particular issue.

\appendix 
\section{Cutting depth cycles: Proof of Lemma \ref{cutSpaceCurves} }\label{proofOfLemCutSpaceCurvesSec}
In this section we will prove Lemma \ref{cutSpaceCurves}. For the reader's convenience, we reproduce it here.
\begin{cutSpaceCurvesLem}
For each $D\geq 1$, there are constants $A = A(D)$ and $\kappa=\kappa(D)$ so that the following holds. 
Let $\C$ be a set of $n$ irreducible algebraic plane curves of degree at most $D$ and let $\hat\C=\{\hat C \mid C\in\C\}$. 
Then by using $\leq A n^{3/2} \log^{\kappa}n$ cuts, $\hat \C$ can be cut into a set of Jordan arcs that have no proper depth cycles of length two. 
\end{cutSpaceCurvesLem}
\begin{proof}
We will make crucial use of Guth's result~\cite{Gut} about polynomial partitioning for varieties. 
The result in \cite{Gut} is fairly general; here we only state the special case that we need.
\begin{theorem}[Polynomial partitioning for varieties, special case]\label{guthPartitionThm}
Let $\mathcal{K}$ be a set of algebraic curves in $\RR^3$ each of which has degree at most $D$. For each $E \geq 1$, there is a non-zero
polynomial $f$ of degree at most $E$ so that each of the $O(E^3)$ connected components of $\RR^3\backslash Z(f)$ intersects
$O_{D}(|\mathcal{K}|/E^2)$ curves of $\mathcal{K}$.
\end{theorem}

Apply Theorem \ref{guthPartitionThm} to $\hat C$ to obtain a partitioning polynomial $f$ of degree at most $E$ 
whose zero set cuts $\RR^3$ into $O(E^3)$ cells (open connected sets), each of which 
intersects $O_D(n/E^{2})$ curves from $\hat\C$.

Let $\hat C$, $\hat C^\prime\in\hat \C$, and let $\gamma\subset\hat C,\ \gamma^\prime\subset\hat C^\prime$ be closed Jordan space arcs 
that form a proper depth cycle of length two. As in the proof of Lemma~\ref{lensesAndDepthCycles},
we assume that $\gamma$ and $\gamma'$ are chosen so that their $xy$-projections have common endpoints and disjoint relative interiors.
Then one of the following must occur:
\begin{enumerate}
 \item[(i)] $\gamma$ and $\gamma^\prime$ are both contained in the same cell of $\RR^3\backslash Z(f)$.
 \item[(ii)] Each of $\gamma$ and $\gamma^\prime$ is contained in a cell, but these cells are distinct. 
 \item[(iii)] At least one of $\gamma,\ \gamma^\prime$ properly intersects $Z(f)$.
 \item[(iv)] Both $\gamma$ and $\gamma^\prime$ are contained in $Z(f)$.
\end{enumerate}

\begin{lemma}\label{cuttingToEliminateiiIv}
We can cut $\hat\C$ into $O_D(E^2n)$ Jordan arcs, so that after the cutting, every depth 2-cycle of type (ii), (iii), and (iv) has been
eliminated. 
\end{lemma}
\begin{proof}
We will begin with depth cycles of type (iii). First, observe that each curve from $\hat\C$ that is not
contained in $Z(f)$ intersects $Z(f)$ in $O_{D}(E)$ points. This follows from the main theorem in \cite{BB} (see also \cite[Theorem A.2]{ST}). Thus if we remove $O_D(En)$ points in total, we have eliminated all depth cycles of type (iii).

We will now describe a procedure that eliminates all depth cycles of type (ii) and (iv). Let $\pi(x,y,z)=(x,y)$ be the projection to the $xy$-plane. 
In the arguments below we will sometimes need to work over the complex numbers, so we will abuse notation slightly and let $\pi$ refer to
both the $xy$-projections $\RR^3\to\RR^2$ and $\CC^3\to\CC^2$.

Define 
\begin{equation*}
 Z_{\operatorname{bad}} = \pi^{-1}\Big(\ \overline{\pi\big(Z_{\CC}(f)\cap Z_{\CC}(\partial_z f)\big)}\ \Big),
\end{equation*}
where $\overline{X}$ denotes the Zariski closure of a set $X$. Write $Z_{\CC}(f)$ as a union of irreducible components $Z_1\cup\cdots\cup Z_\ell$ and write $Z_{\CC}(\partial_z f)$ as a union of irreducible components $Y_1\cup\cdots\cup Y_m$. Each of these components is a surface, and 
\begin{equation*}
\sum_{i=1}^\ell \deg Z_i = \deg Z_{\CC}(f)\leq E,\quad\sum_{i=1}^\ell \deg Y_i = \deg Z_{\CC}(\partial_z f)\leq E.
\end{equation*}
For each index $1\leq i\leq \ell$ and $1\leq j\leq m$, either $Z_i=Y_j$, or $Z_i\cap Y_j$ is an algebraic curve of degree at most $(\deg Z_i)(\deg
Y_j)$. If the former occurs then $\partial_z f$ vanishes on $Z_i$, and this implies that $Z_i=Z_{\CC}(g)$ for some polynomial $g(x,y,z)=g(x,y)$. This
means that $Z_i$ is a vertical cylinder above a the curve $\{(x,y)\in\CC^2\mid g(x,y)=0\}$ in the $xy$-plane. 

We conclude that $Z_{\CC}(f)\cap Z_{\CC}(\partial_z f)\subset\CC^3$ is a union of a complex algebraic space curve of
degree $\leq E^2$ (the union of all curves of the form $Z_i\cap Y_j$ with $Z_i\neq Y_j$) and a vertical cylinder over an algebraic plane curve of degree at most $E$ (the union of all surfaces $Z_i$ with $Z_i=Y_j$ for some indices $i,j$). Thus $\overline{\pi(Z_{\CC}(f)\cap Z_{\CC}(\partial_z f))}$ is a plane curve of degree at most $E^2 + E\leq 2E^2$, and therefore $ Z_{\operatorname{bad}} \subset\CC^3$ is an algebraic variety of degree at most $2E^2$, which is a vertical cylinder over a plane curve of degree at most $2E^2$. 

Define 
\begin{equation*}
 \hat\C_{\operatorname{bad}} = \{\hat C\in\hat\C \mid \hat C\subset Z_{\operatorname{bad}}\}.
\end{equation*}
Here we abuse notation slightly and say that the set $\hat C\subset\RR^3$ is contained in $Z_{\operatorname{bad}}\subset\CC^3$ if the embedding of $\hat C$ into $\CC^3$ is contained in
$Z_{\operatorname{bad}}$ (alternatively, if $\hat C$ is contained in the real locus of $Z_{\operatorname{bad}}$).

Since $\pi(\hat C)\cap\pi(\hat C^\prime)$ is finite for every pair of curves $\hat C,\hat C^\prime\in\hat\C$ (indeed, the intersection is just the set $C\cap C^\prime$, which by B\'ezout's
theorem has cardinality at most $D^2$), we have that the $xy$-projection of each curve in $\hat\C_{\operatorname{bad}}$ is contained in a distinct irreducible component of
$\pi(Z_{\operatorname{bad}})$. Since the latter is an algebraic plane curve of degree at most $2E^2$, we conclude that $| \hat\C_{\operatorname{bad}}|\leq 2E^2$. Cut each of the curves in
$\hat\C_{\operatorname{bad}}$ at every point where they pass above or below another curve from $\hat\C$. Doing so adds $O_D(E^2 n)$ cuts in total.

For each curve $\hat C\in\hat\C\backslash \hat\C_{\operatorname{bad}}$, add a cut at each point of $\hat C\cap Z_{\operatorname{bad}}$. For each of these curves, we add $O_D(E^2)$ cuts, so in
total there are $O_D(E^2n)$ cuts of this kind too. After cutting each curve in $\hat C$ an additional $O_D(1)$ times, we can ensure that the resulting connected components are smooth Jordan
space arcs, and the tangent vector of these arcs never points in the $z$-direction (i.e., the tangent vector is never parallel to the vector $(0,0,1)$). 

We claim that all cycles of type (ii) and (iv) have been eliminated. For each point $p=(x_0,y_0,z_0)\in\RR^3$, let 
\begin{equation*}
\rho_-(p)=\{(x_0,y_0,z)\in\RR^3 \mid z\leq z_0\}.
\end{equation*}
This is the closed ray emanating from $p$ in the negative $z$-direction. Let $h(p)$ be the number of intersection points of $\rho_-(p)$ with $Z(f)$. Observe that $h(p)\geq 0$, and $h(p)$ is
finite unless $p\in Z_{\operatorname{bad}}$ (indeed, if $h(p)$ is infinite, then $f$ must vanish on the vertical line passing through $p$, so in particular, $p\in Z_{\CC}(\partial_zf)\cap Z_{\CC}(f)$).

Note that if $p=(x,y,z)$, $p^\prime=(x,y,z^\prime)$ are points with $z^\prime>z$, then $h(p^\prime)\geq h(p)$. Furthermore, if $p$ and $p^\prime$ lie in different cells of $\RR^3\backslash Z(f)$, then 
\begin{equation}
h(p^\prime)>h(p).
\end{equation}
Similarly, if both $p$ and $p^\prime$ lie in $Z(f)\backslash Z_{\operatorname{bad}}$, then 
\begin{equation}
h(p^\prime)>h(p).
\end{equation}
(indeed, the point $p'$ is counted in $h(p')$ but not in $h(p)$).

Now, let $\hat C,\hat C^\prime$ be curves from $\hat\C\backslash\hat\C_{\operatorname{bad}}$ that form a proper depth cycle of length two of type (ii) or (iv). In particular, this means there
are points $(x_1,y_1,z_1),(x_2,y_2,z_2)\in \hat C,\ (x_1,y_1,z_1^\prime),(x_2,y_2,z_2^\prime)\in\hat C^\prime,$ with $z_1>z_1^\prime$ and $z_2<z_2^\prime.$ By what has just been argued, we
then have $h(z_1)>h(z_1^\prime)$ and $h(z_2)<h(z_2^\prime)$.

In particular, if $\gamma\subset\hat C$ and $\gamma^\prime\subset\hat C^\prime$ are the Jordan arcs with endpoints $(x_1,y_1,z_1),(x_2,y_2,z_2)$ and $(x_1,y_1,z_1^\prime)$ and
$(z_2,y_2,z_2^\prime)$, respectively, then $h(p)$ cannot be constant on both $\gamma$ and $\gamma^\prime$. This will imply, however, that either $\gamma$ or $\gamma^\prime$ must intersect $Z_{\operatorname{bad}}$, as the following lemma shows.

\begin{lemma}
Let $\alpha\subset\RR^3$ be a smooth simple curve and let $p_0\in\alpha$. Suppose that the tangent vector to $\alpha$ at $p_0$ does not point in the vertical direction. If $p_0\not\in Z_{\operatorname{bad}}$, then $h(p)$ is constant for all $p\in\alpha$ in a small neighborhood of $p_0$.
\end{lemma}
\begin{proof}
In short, the result follows from the implicit function theorem. Restricting $\alpha$ to a sufficiently small neighborhood of $p_0$ if necessary, we can assume that the projection
$\pi(\alpha)$ is a simple smooth plane curve. Thus the set $\pi^{-1}(\pi(\alpha))$ is a vertical cylindrical strip above the plane curve $\pi(\alpha)$. By further restricting $\alpha$ if
necessary, we have that $Z(f)\cap\pi^{-1}(\pi(\alpha))$ is a pairwise disjoint union of simple smooth curves, and at every point on each of these curves, the tangent vector of the curve does not point in the vertical direction. This is because every point of $Z(f)\cap \pi^{-1}(\pi(\alpha))$ is a smooth point of $Z(f)$, and $Z(f)$ intersects $\pi^{-1}(\pi(\alpha))$ transversely at every point of $Z(f)\cap \pi^{-1}(\pi(\alpha))$. 

In particular, this means that by further restricting $\alpha$ if necessary, a vertical line passing through $\alpha$ intersects each of these curves in exactly one point. We conclude that $h(p)$ is constant for all $p\in\alpha$.
\end{proof}
\begin{corollary}
If we remove every point $p$ from each curve $\hat C\in\hat\C\backslash\hat\C_{\operatorname{bad}}$ for which either $p\in Z_{\operatorname{bad}}$ or $\hat C$ has a vertical tangent vector at
$p$, then we have eliminated all depth cycles of type (ii) and (iv).
\end{corollary}
Note that in eliminating the depth cycles of type (ii), (iii), and (iv), we have made a total of $O_D(E^2n)$ cuts. This concludes the proof of Lemma \ref{cuttingToEliminateiiIv}.
\end{proof}

We are now ready to prove Lemma \ref{cutSpaceCurves}. We will do so using induction on $n$. When $n$ is smaller than some constant threshold $n_0$ (that the following analysis will implicitly
specify),
the result is immediate. Indeed, we can cut each curve $\hat C\in\hat\C$ at each point whose $xy$-projection is also incident to the $xy$-projection of another curve, and this requires at most
$D^2\binom{n}{2} = O(D^2n^2)$ cuts. The latter quantity is bounded by $An^{3/2}\log^\kappa n$, provided we choose $A$ sufficiently large (depending on $n_0$). 

Now suppose the result has been proved for all collections of curves of cardinality at most $n-1$. Let $E=n^{1/4}$ (the exact value of $E$ does not matter; any fixed fractional power not
larger than $n^{1/4}$ would suffice). Applying Lemma \ref{cuttingToEliminateiiIv}, we can eliminate all
depth cycles of type (ii)--(iv) using $O_D(E^2n)=O_D(n^{3/2})$ cuts. It remains to consider cycles of type (i). However, each cell contains $O_D(n/E^2)$ curves from $\hat\C$. Thus if we apply the induction hypothesis inside each cell, we conclude that the number of cuts needed to eliminate all cycles of type (i) is 

\begin{equation*}
  O(E^3) \cdot \Big(A\ O_{D}(n/E^{2})^{3/2} \log^\kappa[O_{D}(n/E^{2})] \Big)\leq C_1 A n^{3/2}(\log n - 2\log E + C_2)^\kappa,
\end{equation*}
where $C_1$ and $C_2$ are constants depending only on $D$. Since $E=n^{1/4}$, we have $\log n - 2\log E = \frac12\log n$. If $n$ is sufficiently large compared to $C_2$, then 
$\log n - \log E + C_2\leq \frac34\log n$. If we select $\kappa=O_D(1)$ sufficiently large (depending on $C_1$), so as to make
$C_1\left(\frac34\right)^\kappa \le \frac12$, the number of cuts needed to eliminate all cycles of type (i) is at most
\begin{equation*}
C_1 A n^{3/2}(\log n - 2\log E + C_2)^\kappa\leq \frac{1}{2} A n^{3/2}\log^\kappa n.
\end{equation*}

Adding up the number of cuts required to eliminate all cycles of types (i)--(iv), we conclude that the total number of cuts needed is at most
\begin{equation}
\frac{1}{2} A n^{3/2}\log^\kappa n + C_3E^2n \le
\frac{1}{2} A n^{3/2}\log^\kappa n + C_3n^{3/2} ,
\end{equation}
where $C_3$ is another constant that depends only on $D$ (recall that $E=n^{1/4}$). Choosing $A>2C_3$, the total number of cuts is at most
\begin{equation*}
\left(\frac12 A + C_3\right) n^{3/2}\log^\kappa n \le
A n^{3/2}\log^\kappa n .
\end{equation*}
This closes the induction and completes the proof.
\end{proof}

\section{Algebraic families of curves versus degrees of freedom}\label{dimFamilyVsDegreesOfFreedom}
In this section we will compare the conditions imposed by Theorem \ref{incidencesPtsCurves} with the conditions from Pach and Sharir's incidence theorem. Since Theorem \ref{incidencesPtsCurves} only applies to algebraic curves, we will restrict our attention to curves of this type. It is worth noting, however, that Pach and Sharir's incidence theorem does not require that the curves be algebraic. 

Pach and Sharir's incidence theorem applies to sets of points $\P\subset\RR^2$  and sets of (algebraic) plane curves $\C$ that have $s$ degrees of freedom and multiplicity type $t$. This means that each pair of curves from $\C$ intersect at most $t$ times, and for any set of $s$ points from $\mathcal{P},$ there are at most $t$ curves from $\C$ containing all $s$ points. The property of having $s$ degrees of freedom (and multiplicity type $t$) depends only on the finite collections $\P$ and $\C$. While $\C$ might be (secretly) taken from some pre-specified family of curves, Pach and Sharir's incidence theorem has no knowledge of this family. 

In practice, however, the curves in $\C$ \emph{are} taken from a pre-specified family of curves, and we usually prove that $(\P,\C)$ has $s$ degrees of freedom by arguing that \emph{any} set of
curves from this family (and any set of points) has $s$ degrees of freedom. The following theorem will show that, under some mild conditions, if $F$ is a family of curves, and if \emph{every} finite set of curves from $F$ has $s$ degrees of freedom, then $F$ must be an $s$-dimensional family of curves. The converse is more subtle and will be discussed further at the end of this section.

Let $F\subset\mathbf{P}\mathbb{R}^{\binom{D+2}{2}}$ be an irreducible variety. We say that $F$ is \emph{degenerate} if there is an irreducible curve $\gamma_0\subset\RR^2$ that is a component
of every curve $\gamma\in F$. Degenerate families are not interesting from the point of view of incidence geometry, since it is possible to select $n$ curves from a degenerate family and $m$
points that yield $mn$ incidences. Generally, we rule out this situation by requiring that no two curves share a common component. If the curves are chosen from a degenerate family then the only way this condition can be satisfied is if there is only one curve.

If $F\subset\mathbf{P}\mathbb{R}^{\binom{D+2}{2}}$ is a (possibly reducible) variety, we say $F$ is degenerate if any of its irreducible components are degenerate. In practice, we can throw away any degenerate components of $F$, since there are $O_{D,\deg(F)}(1)$ degenerate components total, and if we require that no two curves share a common component then at most one curve can be chosen from each degenerate component of $F$.

\begin{theorem}\label{degFreedomImpliesDimensionOfF}
Let $F\subset\mathbf{P}\mathbb{R}^{\binom{D+2}{2}}$ be a non-degenerate family of curves. Suppose that there exist constants $s$ and $t$ so that for every finite set $\mathcal{P}\subset\RR^2$
and every finite set of curves $\C\subset F$, no two of which share a common component, the collection $(\mathcal{P},C)$ has $s$ degrees of freedom and multiplicity type $t$. Then $\dim(F)\leq
s$, i.e. the dimension of $F$ is at most the (worst-case) number of degrees of freedom of a collection of curves taken from $F$ (with respect to any finite point set $\P$). 
\end{theorem}
\begin{proof}
Suppose $\dim(F)>s$. For each $t$, we will find a set of $s$ points in $\RR^2$ and a set $\C$ of $t$ curves from $\mathcal{F}$ so that each of the $s$ points is incident to every curve from $\C$. 

Let $w\in F$ be a smooth point (in dimension $\dim(F)$) corresponding to a curve $\gamma_0\subset\RR^2$. The precise definition of a smooth point in dimension $\dim(F)$ is given in \cite[\S 3.3]{BCR}. The only property that we will use, however, is that there is a small Euclidean ball $B(w,\eps)\subset\mathbf{P}\mathbb{R}^{\binom{D+2}{2}}$ so that $F\cap B(w,\eps)$ is a smooth manifold of dimension $\dim(F)$ (again, see \cite[\S 3.3]{BCR} for details).

For each $p\in \RR^2$, let $H_p=\{\gamma\in \mathbf{P}\mathbb{R}^{\binom{D+2}{2}} \mid p\in\gamma\}$. We claim that there exists a point $p\in\RR^2$ so
that $B(w,\eps)\cap F\cap H_p$ is a smooth manifold of dimension $\dim(F)-1$. Indeed, since the irreducible component of $F$ containing $w$ is
non-degenerate, $\bigcap_{q\in\gamma_0}H_q$ only contains the point $w$.

Thus there exists a hyperplane $H_q$ that is transverse to $F$ at $w$, so (by making $\eps$ smaller if necessary), $B(w,\eps)\cap H_q\cap F$ is a smooth manifold of dimension $\dim(F)-1$. Repeat this $s$ times; we obtain a smooth manifold $B(w,\eps)\cap F\cap H_{q_1}\cap H_{q_2}\cap\ldots\cap H_{q_s}$ of dimension $\geq 1$. Since the irreducible component of $F$ containing $w$ is non-degenerate, we can select $t+1$ points from this manifold that correspond to curves in $\RR^2$, no two of which have a common irreducible component. Each of these curves is incident to the points $q_1,\ldots,q_s\in\RR^2$. 

This contradicts the assumption that for every finite set $\mathcal{P}\subset\RR^2$ and every finite set of curves $\C\subset F$, no two of which share a common component, the collection $(\mathcal{P},C)$ has $s$ degrees of freedom and multiplicity type $t$. We conclude that $\dim(F)\leq s$. 
\end{proof}

The converse direction is a bit more subtle. For example, let $p_0\in\RR^2$ and let $F\subset\mathbf{P}\mathbb{R}^{\binom{D+2}{2}}$ be the family of all curves of degree at most $D$ that pass through the point $p_0$. Then $\dim(F) = \binom{D+2}{2}-2$, but for each $t$, it is possible to find a finite set of points $\P$ and a set of irreducible curves $\C\subset F$ so that $(\P,\C)$ fails to have $\binom{D+2}{2}-2$ degrees of freedom and multiplicity type $t$. However, such an example is ``cheating,'' since $\P$ must contain the point $p_0$. If we require that $p_0\not\in\P$, then $(\P,\C)$ will have $\binom{D+2}{2}-2$ degrees of freedom and multiplicity type $O_D(1)$. However, we can disguise the previous example in more subtle ways. Thus it appears difficult to determine under what conditions a converse to Theorem \ref{degFreedomImpliesDimensionOfF} could be expected to hold.

\end{document}

%% file: cut3int.pstex_t
\begin{picture}(0,0)%
\includegraphics{cut3int.pstex}%
\end{picture}%
\setlength{\unitlength}{3947sp}%
\begingroup\makeatletter\ifx\SetFigFont\undefined%
\gdef\SetFigFont#1#2#3#4#5{%
  \reset@font\fontsize{#1}{#2pt}%
  \fontfamily{#3}\fontseries{#4}\fontshape{#5}%
  \selectfont}%
\fi\endgroup%
\begin{picture}(2683,2796)(429,-2138)
\end{picture}

%% file: figlenses.pstex_t
\begin{picture}(0,0)%
\includegraphics{figlenses.pstex}%
\end{picture}%
\setlength{\unitlength}{3729sp}%
\begingroup\makeatletter\ifx\SetFigFont\undefined%
\gdef\SetFigFont#1#2#3#4#5{%
  \reset@font\fontsize{#1}{#2pt}%
  \fontfamily{#3}\fontseries{#4}\fontshape{#5}%
  \selectfont}%
\fi\endgroup%
\begin{picture}(4695,1465)(893,-2027)
\put(908,-1561){\makebox(0,0)[lb]{\smash{{\SetFigFont{11}{13.2}{\rmdefault}{\mddefault}{\updefault}{\color[rgb]{0,0,0}$\gamma$}%
}}}}
\put(931,-721){\makebox(0,0)[lb]{\smash{{\SetFigFont{11}{13.2}{\rmdefault}{\mddefault}{\updefault}{\color[rgb]{0,0,0}$\gamma'$}%
}}}}
\put(1831,-841){\makebox(0,0)[lb]{\smash{{\SetFigFont{11}{13.2}{\rmdefault}{\mddefault}{\updefault}{\color[rgb]{0,0,0}$\delta$}%
}}}}
\put(1808,-1396){\makebox(0,0)[lb]{\smash{{\SetFigFont{11}{13.2}{\rmdefault}{\mddefault}{\updefault}{\color[rgb]{0,0,0}$\delta'$}%
}}}}
\put(1755,-1958){\makebox(0,0)[lb]{\smash{{\SetFigFont{11}{13.2}{\rmdefault}{\mddefault}{\updefault}{\color[rgb]{0,0,0}(a)}%
}}}}
\put(4644,-1958){\makebox(0,0)[lb]{\smash{{\SetFigFont{11}{13.2}{\rmdefault}{\mddefault}{\updefault}{\color[rgb]{0,0,0}(b)}%
}}}}
\put(3705,-1351){\makebox(0,0)[lb]{\smash{{\SetFigFont{11}{13.2}{\rmdefault}{\mddefault}{\updefault}{\color[rgb]{0,0,0}$\gamma$}%
}}}}
\put(3886,-848){\makebox(0,0)[lb]{\smash{{\SetFigFont{11}{13.2}{\rmdefault}{\mddefault}{\updefault}{\color[rgb]{0,0,0}$\gamma'$}%
}}}}
\put(4508,-698){\makebox(0,0)[lb]{\smash{{\SetFigFont{11}{13.2}{\rmdefault}{\mddefault}{\updefault}{\color[rgb]{0,0,0}$\delta$}%
}}}}
\put(5573,-961){\makebox(0,0)[lb]{\smash{{\SetFigFont{11}{13.2}{\rmdefault}{\mddefault}{\updefault}{\color[rgb]{0,0,0}$\delta'$}%
}}}}
\put(1208,-1254){\makebox(0,0)[lb]{\smash{{\SetFigFont{11}{13.2}{\rmdefault}{\mddefault}{\updefault}{\color[rgb]{0,0,0}$u$}%
}}}}
\put(2341,-1246){\makebox(0,0)[lb]{\smash{{\SetFigFont{11}{13.2}{\rmdefault}{\mddefault}{\updefault}{\color[rgb]{0,0,0}$v$}%
}}}}
\put(3975,-1171){\makebox(0,0)[lb]{\smash{{\SetFigFont{11}{13.2}{\rmdefault}{\mddefault}{\updefault}{\color[rgb]{0,0,0}$u$}%
}}}}
\put(4891,-908){\makebox(0,0)[lb]{\smash{{\SetFigFont{11}{13.2}{\rmdefault}{\mddefault}{\updefault}{\color[rgb]{0,0,0}$v$}%
}}}}
\end{picture}%

%% file: lift.pstex_t
\begin{picture}(0,0)%
\includegraphics{lift.pstex}%
\end{picture}%
\setlength{\unitlength}{3729sp}%
\begingroup\makeatletter\ifx\SetFigFont\undefined%
\gdef\SetFigFont#1#2#3#4#5{%
  \reset@font\fontsize{#1}{#2pt}%
  \fontfamily{#3}\fontseries{#4}\fontshape{#5}%
  \selectfont}%
\fi\endgroup%
\begin{picture}(4205,1128)(796,-1585)
\put(1766,-646){\makebox(0,0)[lb]{\smash{{\SetFigFont{11}{13.2}{\rmdefault}{\mddefault}{\updefault}{\color[rgb]{0,0,0}$\gamma$}%
}}}}
\put(1666,-1191){\makebox(0,0)[lb]{\smash{{\SetFigFont{11}{13.2}{\rmdefault}{\mddefault}{\updefault}{\color[rgb]{0,0,0}$\gamma'$}%
}}}}
\put(1076,-1051){\makebox(0,0)[lb]{\smash{{\SetFigFont{11}{13.2}{\rmdefault}{\mddefault}{\updefault}{\color[rgb]{0,0,0}$p$}%
}}}}
\put(2266,-1046){\makebox(0,0)[lb]{\smash{{\SetFigFont{11}{13.2}{\rmdefault}{\mddefault}{\updefault}{\color[rgb]{0,0,0}$q$}%
}}}}
\put(1561,-1516){\makebox(0,0)[lb]{\smash{{\SetFigFont{11}{13.2}{\rmdefault}{\mddefault}{\updefault}{\color[rgb]{0,0,0}(a)}%
}}}}
\put(3931,-1516){\makebox(0,0)[lb]{\smash{{\SetFigFont{11}{13.2}{\rmdefault}{\mddefault}{\updefault}{\color[rgb]{0,0,0}(b)}%
}}}}
\put(3961,-616){\makebox(0,0)[lb]{\smash{{\SetFigFont{11}{13.2}{\rmdefault}{\mddefault}{\updefault}{\color[rgb]{0,0,0}$\hat{\gamma}$}%
}}}}
\put(4056,-1196){\makebox(0,0)[lb]{\smash{{\SetFigFont{11}{13.2}{\rmdefault}{\mddefault}{\updefault}{\color[rgb]{0,0,0}$\hat{\gamma}'$}%
}}}}
\end{picture}%